\documentclass[12pt]{amsart}

\usepackage[matrix,arrow,curve,frame]{xy}
\usepackage{amsmath,amsthm,amssymb,enumerate}
\usepackage{latexsym}
\usepackage{amscd}
\usepackage[colorlinks=false]{hyperref}
\usepackage{euscript}

\newtheorem{thm}[subsection]{Theorem}
\newtheorem{defn}[subsection]{Definition}
\newtheorem{prop}[subsection]{Proposition}

\newtheorem{cor}[subsection]{Corollary}
\newtheorem{lemma}[subsection]{Lemma}

\newtheorem{remark}[subsection]{Remark}

\theoremstyle{definition}

\numberwithin{equation}{section}

\def\bpartial{{\bar\partial}}

\def\cE{\mathcal E}
\def\cF{\mathcal F}

\def\cI{\mathcal I}
\def\cJ{\mathcal J}

\def\cM{\mathcal M}
\def\cN{\mathcal N}

\def\cS{\mathcal S}
\def\cT{\mathcal T}

\def\cW{\mathcal W}
\def\R{\mathfrak R}

\def\B{\mathfrak B}

\def\Zplus{\mathbb Z_{\geq 0}}

\DeclareMathOperator{\sign}{sign}

\DeclareMathOperator{\Hom}{Hom}

\DeclareMathOperator{\Spec}{Spec}

\newfont{\german}{eufm10}

\begin{document}
\pagestyle{plain}

\title{Standard monomials and invariant theory of arc spaces II: Symplectic group }

\author{Andrew R. Linshaw}
\address{Department of Mathematics, University of Denver}
\email{andrew.linshaw@du.edu}
\thanks{A. Linshaw is supported by Simons Foundation Grant \#635650 and NSF Grant DMS-2001484.}

\author{Bailin Song}
\address{School of Mathematical Sciences, University of Science and Technology of China, Hefei, Anhui 230026, P.R. China}
\email{bailinso@ustc.edu.cn}
\thanks{B. Song is supported by  NSFC No. 12171447}

\begin{abstract}
This is the second in a series of papers on standard monomial theory and invariant theory of arc spaces. For any algebraically closed field $K$, we construct a standard monomial basis for the arc space of the Pfaffian variety over $K$. As an application, we prove the arc space analogue of the first and second fundamental theorems of invariant theory for the symplectic group. \end{abstract}


\maketitle
\section{Introduction}

\subsection{Invariant theory} 
Given an algebraically closed field $K$, an algebraic group $G$ over $K$, and a finite-dimensional $G$-module $W$, a fundamental problem in invariant theory is to describe the ring of invariant polynomial functions $K[W]^G$. It is natural to also consider $K[V]^G$, where $V=W^{\oplus p}\bigoplus {W^*}^{\oplus q}$ is the direct sum of $p$ copies of $W$ and $q$ copies of the dual $G$-module $W^*$. In Weyl's terminology \cite{W}, a {\it first fundamental theorem of invariant theory} (FFT) for the pair $(G,W)$ is a generating set for $K[V]^G$, and a {\it second fundamental theorem} (SFT) for $(G,W)$ is a generating set for the ideal of relations among the generators of $K[V]^G$. When $\text{char} \ K = 0$, if $G$ is one of the classical groups and $W$ is the standard module, the FFTs and SFTs are due to Weyl \cite{W}. The analogous results in arbitrary characteristic were proven much later by de Concini and Procesi using standard monomial theory \cite{DCP}.

In this paper, we consider the case where $G$ is the symplectic group. For an even integer $h$, let $W=K^{\oplus h}$ be equipped with a non-degenerate, skew-symmetric bilinear form given by $w=\sum_{i=1}^{h/2} dz_{2i-1}\wedge dz_{2i}$. Then $$Sp_h(K)=\{A \in SL_h(K)|\ A\  \text{ preserves}\  w\}$$ is the symplectic group over $K$. For $p\geq 1$, let $V=W^{\oplus p}$ be the direct sum of $p$ copies of $W$. The affine coordinate ring of $V$ is 
 $$K[V]=K[a^{(0)}_{il} |\ 1\leq i\leq p, \ 1\leq l\leq h].$$ 
\begin{thm} (FFT and SFT for $Sp_h(K)$ and $W = K^{\oplus h}$) \label{classicalFFTSFT}
\begin{enumerate}
\item The ring of invariants $K[V]^{Sp_h(K)}$ is generated by 
 \begin{equation*} \label{generators:classical} X^{(0)}_{uv}= \sum_{i=1}^{h/2} (a^{(0)}_{u{2i-1}}a^{(0)}_{v2i}-a^{(0)}_{v{2i-1}}a^{(0)}_{u2i}),\qquad 1\leq u,v \leq p. \end{equation*}
\item The ideal of relations among the generators  in (1) is generated by the Pfaffians
\begin{equation}\label{relations:classical} P\left(
\begin{array}{cccc}
X^{(0)}_{u_1u_1} & X^{(0)}_{u_1u_2} & \cdots &X^{(0)}_{u_1u_{h+2}}\\
X^{(0)}_{u_2u_1}& X^{(0)}_{u_2u_2} & \cdots & X^{(0)}_{u_2u_{h+2}} \\
\vdots & \vdots & \vdots & \vdots \\
X^{(0)}_{u_{h+2}u_1} & X^{(0)}_{u_{h+2}u_2} & \cdots & X^{(0)}_{u_{h+2}u_{h+2}} \\
\end{array}
\right),\qquad 1\leq u_i<u_{i+1}\leq p.
\end{equation}
\end{enumerate}
\end{thm}

\subsection{Standard monomials and Pfaffian varieties} 
Standard monomial theory was initiated by Seshadri, Musili and Lakshmibai \cite{Se, LS,LMS1,LMS2}, generalizing earlier work of Hodge \cite{H}. It involves combinatorial bases for the coordinate rings of Schubert varieties inside quotients of classical groups by parabolic subgroups. In this paper, we only need the case of Pfaffian varieties. For a positive integer $p$, let \begin{equation} \label{def:R} R= R^p = \mathbb Z[x^{(0)}_{uv}|\ 1\leq u,v\leq p]/(x^{(0)}_{uv}+x^{(0)}_{vu},x^{(0)}_{uu})\end{equation} be the ring of polynomial functions with integer coefficients on the space of $p\times p$ skew-symmetric matrices. Consider the Pfaffian $P(B)$ of the skew-symmetric matrix
\begin{equation}\label{eqn:minor} B=\left(
\begin{array}{cccc}
x^{(0)}_{u_1u_1} & x^{(0)}_{u_1u_2} & \cdots &x^{(0)}_{u_1u_h}\\
x^{(0)}_{u_2u_1}& x^{(0)}_{u_2u_2} & \cdots & x^{(0)}_{u_2u_h} \\
\vdots & \vdots & \vdots & \vdots \\
x^{(0)}_{u_hu_1} & x^{(0)}_{u_hu_2} & \cdots & x^{(0)}_{u_hu_h} \\
\end{array}
\right),
\end{equation}
with $1\leq u_i<u_{i+1}\leq p$ and $h\in 2\Zplus$. Throughout this paper, we will represent $P(B)$ by the ordered $h$-tuple $|u_{h},\dots,u_2,u_1|$. 
There is a partial ordering on the set of these Pfaffians given by
$$|u_h,\dots,u_2,u_1|\leq |u'_{h'},\dots,u'_2,u'_1|, \quad \text{if } h'\leq h, u_i\leq u'_i.$$ 
$R$ has a standard monomial basis (cf.~\cite{LR}) with
respect to this partially ordered set; the ordered products $A_1A_2\cdots A_k$ of Pfaffians $A_i$ with $A_i\leq A_{i+1}$, form a basis of $R$.

Similarly, let $R[h]$ be the ideal of $R$ generated by the Pfaffians of the diagonal $h\times h$-minors, which are precisely the elements of the form \eqref{relations:classical} with $h+2$ replaced by $h$. Let 
\begin{equation} \label{def:Rh} R_h=R\slash R[h+2]. \end{equation} Then $R_h$ has a basis consisting of ordered products $A_1A_2\cdots A_k$ of the Pfaffians $A_i$ with $|h,\dots,2,1| \leq   A_i\leq A_{i+1}$.

For an arbitrary algebraically closed field $K$, let $SM_p = SM_{p}(K)$ be the affine space of $p\times p$ skew-symmetric matrices with entries in $K$. The affine coordinate ring $K[SM_p]$ is obtained from $R$ by base change, that is, $K[SM_p] = R\otimes_{\mathbb{Z}} K$. Let $K[SM_p][h]$ be the ideal generated by the Pfaffians of the diagonal $h\times h$ minors. The Pfaffian variety $Pf_h = Pf_h(K)$ is a closed subvariety of $SM_{p}$ with $K[SM_p][h]$ as the defining ideal. Then for an even integer $h$, the affine coordinate ring $K[Pf_{h}] = K[SM_p]/ K[SM_p][h]$ has a standard monomial basis: the ordered products of  $A_1A_2\cdots A_k$ with $|h-2,\dots,2,1|\leq   A_i\leq A_{i+1}$ form a basis of $K[Pf_{h}]$. For $G = Sp_h(K)$ and $V$ as in Theorem \ref{classicalFFTSFT}, we have $V/\!\!/Sp_k(K) = \text{Spec} \ K[V]^{Sp_h(K)} \cong Pf_{h+2}$. The proof of Theorem \ref{classicalFFTSFT} in \cite{DCP} makes use of this standard monomial basis. The key is to consider an integral form of $K[V]$, namely, $\mathbb{Z}[a_{il}^{(0)}]$ for $1\leq i\leq p$ and $1\leq l\leq h$, and show that the natural map $R_h \rightarrow \mathbb{Z}[a_{il}^{(0)}]$ is injective. After tensoring with $K$, this yields an injective map $K[Pf_{h+2}] \rightarrow K[V]$ whose image is precisely $K[V]^{Sp_h(K)}$.  A uniform treatment of these results for all the classical results using standard monomial theory can be found in the book \cite{LR} of Lakshmibai and Raghavan.

\subsection{Arc spaces} Given an irreducible scheme $X$ of finite type over $K$, the arc space $J_\infty(X)$ is defined as the inverse limit of the finite jet schemes $J_n(X)$ \cite{EM}. By Corollary 1.2 of \cite{B}, it is determined by its functor of points: for every $K$-algebra $A$, we have a bijection
$$\Hom(\Spec A, J_\infty(X))\cong\Hom(\Spec A[[t]], X).$$ If $G$ is an algebraic group over $K$, $J_{\infty}(G)$ is again an algebraic group over $K$. If $V$ is a finite-dimensional $G$-module, there is an induced action of $J_{\infty}(G)$ on $J_{\infty}(V)$. The quotient morphism $V\to V/\!\!/G$ induces a morphism $J_\infty(V)\to J_\infty(V/\!\!/G)$, so we have a morphism 
\begin{equation} \label{equ:invmap} J_\infty(V)/\!\!/J_\infty(G)\to J_\infty(V/\!\!/G),\end{equation} and the corresponding ring homomorphism
\begin{equation} \label{equ:invringmap} K[J_{\infty}(V/\!\!/G)] \rightarrow K[J_{\infty}(V)]^{J_{\infty}(G)}.\end{equation}  In the case $K=\mathbb C$, if $G$ is connected and $V/\!\!/G$ is smooth, it was shown in \cite{LSSI} that \eqref{equ:invringmap} is an isomorphism, and under some additional hypotheses this also holds when $V/\!\!/G$ is a complete intersection. In general, \eqref{equ:invringmap} is neither injective nor surjective.

 \subsection{Standard monomials for arc spaces} Let 
 \begin{equation} \label{eqn:refofRR} \R=\R^{p}=\mathbb Z[x^{(k)}_{uv}|\ 1\leq u,v\leq p,\ k\geq 0]/(x^{(k)}_{uv}+x^{(k)}_{vu}, x^{(k)}_{uu}),\end{equation} with a derivation $\partial$ characterized by $\partial x_{uv}^{(k)}=(k+1)x_{uv}^{(k+1)}$. It can be regarded as the ring of polynomial functions with integer coefficients on the arc space of $p\times p$ skew-symmetric matrices; in particular, $K[J_{\infty}(SM_p)] \cong \R \otimes_{\mathbb{Z}} K$.

Let $\R[h]$ be the ideal of $\R$ generated by the Pfaffians of the diagonal $h\times h$ minors in the form of (\ref{eqn:minor}) and their normalized derivatives $\frac 1 {n!}\partial^n P(B)$.  Let $\R_h=\R\slash\R[h+2]$. Let $\cJ_r$ be the set of Pfaffians of the matrices of the form of (\ref{eqn:minor}) with $h\leq r$ and their normalized derivatives $\frac 1 {n!}\partial^n P(B)$. Note that $R$ and $R_h$ are naturally subrings of $\R$ and $\R_h$, respectively. In Section \ref{sect:standardmono}, we will define a notion of standard monomial on $\cJ_h$ that extends the above notion on $R_h$, and in Section \ref{section:3} we will prove the following result.

\begin{thm}\label{thm:standard}For an even integer $h$, $\R_h$ has a $\mathbb Z$-basis given by the standard monomials of $\cJ_h$.
\end{thm}
The proof is based on a technical result (Lemma \ref{lem:base0}) whose proof is quite long and is deferred to Section \ref{section:proof2.5}.

Let $J_\infty(Pf_{h+2})$ be the arc space of the Pfaffian variety $Pf_{h+2}$. Then the affine coordinate ring $K[J_\infty(Pf_{h+2})]$ is $\R_{h}\otimes_{\mathbb{Z}} K$, so we immediately have the following corollary.
 \begin{cor}\label{cor:standarddeterm} $K[J_\infty(Pf_{h+2})]$ has a $K$-basis given by the standard monomials of $\cJ_{h}$. 
  \end{cor} 
  
\subsection{Application in invariant theory}
Our main application of Theorem \ref{thm:standard} is to prove the arc space analogue of Theorem \ref{classicalFFTSFT}. As above, for an even integer $h \geq 2$, let $Sp_h(K)$ be the symplectic group over $K$, $W=K^{\oplus h}$ its standard representation, and $V=W^{\oplus p}$ the sum of $p$ copies of $W$. Then $$K[J_{\infty}(V)]=K[a^{(k)}_{il}|\ 1\leq i\leq p, \ 1\leq l\leq h, \ k\in \Zplus],$$ which has an induced action of $J_\infty(Sp_h(K))$. The following theorem is the arc space analogue of Theorem \ref{classicalFFTSFT}, and is proved in Section \ref{section:4}.
\begin{thm}\label{thm:main}
	Let $X^{(k)}_{uv}=\bpartial^k\sum_{i=1}^{h/2} (a^{(0)}_{u{2i-1}}a^{(0)}_{v2i}-a^{(0)}_{v{2i-1}}a^{(0)}_{u2i})$, for $1\leq u,v \leq p$, where $\bpartial^k = \frac{1}{k!} \partial$.
	\begin{enumerate}
		\item The ring of invariants $K[J_{\infty}(V)]^{J_\infty(Sp_h(K))}$ is generated by $X^{(k)}_{uv}$.
		\item The ideal of relations among the generators  in (1) is generated by 
		\begin{equation}\label{eqn:minorX} \bpartial^k P\left(
		\begin{array}{cccc}
		X^{(0)}_{u_1u_1} & X^{(0)}_{u_1u_2} & \cdots &X^{(0)}_{u_1u_{h+2}}\\
		X^{(0)}_{u_2u_1}& X^{(0)}_{u_2u_2} & \cdots & X^{(0)}_{u_2u_{h+2}} \\
		\vdots & \vdots & \vdots & \vdots \\
		X^{(0)}_{u_{h+2}u_1} & X^{(0)}_{u_{h+2}u_2} & \cdots & X^{(0)}_{u_{h+2}u_{h+2}} \\
		\end{array}
		\right), \qquad 1\leq u_i<u_{i+1}\leq p.
		\end{equation}
		\item  $K[J_{\infty}(V)]^{J_\infty(Sp_h(K))}$ has a $K$-basis given by standard monomials of $\cJ_h$.
	\end{enumerate}
\end{thm}

\begin{cor}\label{cor:quot} For all $h \geq 1$ and $p\geq 1$, the map $K[J_{\infty}(V/\!\!/Sp_h(K))] \rightarrow K[J_{\infty}(V)]^{J_{\infty}(Sp_h(K))}$ given by \eqref{equ:invringmap}, is an isomorphism. In particular,
$$J_\infty(V)/\!\!/J_\infty(Sp_h(K))\cong J_\infty(V/\!\!/Sp_h(K)).$$
\end{cor}
Corollary \ref{cor:quot} generalizes Theorem 4.5 of  \cite{LSSI}, which is the case $K = \mathbb{C}$ and $p \leq h+2$. A similar result was proven in \cite{LS1} for the general linear group $GL_k(K)$. The approach in this paper is similar to \cite{LS1} but more involved since we need a result of Bardsley and Richardson \cite{BR} which provides a version of the Luna slice theorem in arbitrary characteristic. 

Theorem \ref{thm:main} has significant applications to vertex algebras which are developed in \cite{LS2,LS3,CLS}. First, it provides a complete description of certain cosets of affine vertex algebras inside free field algebras that are related to the classical Howe pairs. This implies the classical freeness of the simple affine vertex (super)algebras $L_k(\mathfrak{osp}_{m|2n})$ for integers $k,m,n \geq 0$ satisfying $-\frac{m}{2} + n+ k + 1 > 0$. Next, for any smooth manifold $X$ in either the algebraic, complex analytic or smooth settings, the {\it chiral de Rham complex}  $\Omega^{\text{ch}}_X$ is a sheaf of vertex algebras on $X$ that was introduced by Malikov, Schechtman and Vaintrob in \cite{MSV}. Theorem \ref{thm:main} is essential in the description of the vertex algebra of global sections $\Gamma(X, \Omega^{\text{ch}}_X)$ for a $d$-dimensional compact K\"ahler manifold $X$ with holonomy group  $Sp(\frac{d}{2})$. This algebra is isomorphic to the simple small $\cN = 4$ superconformal algebra with central charge $3d$, and is an important building block in the structure of $\Gamma(X, \Omega^{\text{ch}}_X)$ for an arbitrary compact Ricci-flat K\"ahler manifold $X$ \cite{LS2}.

\section{Standard monomials} \label{sect:standardmono}

Fix an integer $p\geq 1$ and recall the ring
$$\R=\mathbb Z[x^{(k)}_{uv}|\ 1\leq u,v\leq p,\ k\geq 0]/(x^{(k)}_{uv}+x^{(k)}_{vu}, x^{(k)}_{uu}),$$ with derivation $\partial$ given on generators by $\partial x_{uv}^{(k)}=(k+1)x_{uv}^{(k+1)}$. As above, this is an integral version of the coordinate ring of the arc space of the space $SM_p$ of $p\times p$ skew-symmetric matrices, i.e., $K[J_{\infty}(SM_p)] = \R \otimes_{\mathbb{Z}} K$, for any field $K$.

For $ l \geq 0$, we have the $l^{\text{th}}$ normalized derivation $\bpartial^l=\frac 1 {l!}\partial^l$ on $\R$. It satisfies $\bpartial^l x_{ij}^{(k)}=C_{k+l}^l  x_{ij}^{(k+l)}\in \R$, where for $k,n\in \mathbb Z_{\geq 0}$, $$C_n^k=\left\{ \begin{array}{cc}
\frac{n!}{k!(n-k)!}, &0\leq k\leq n;\\
0, &\text{otherwise.}
\end{array}\right.
$$

The following propositions are easy to verify.
\begin{prop} For any $a,b\in\R$,
 $$\bpartial^l(ab)=\sum_{i=0}^l\bpartial^ia\, \bpartial^{l-i}b,$$
and $\bpartial^l a\in \R$.
\end{prop}
\begin{prop}\label{prop:detexpression}
	For a skew-symmetric matrix $B$ of the form in Equation (\ref{eqn:minor}) with $h=2l$,
		\begin{equation}\bpartial^n P(B)=\sum_{\substack{n_1+\cdots+n_l=n\\n_i\in \mathbb Z_{\geq 0}}}\sum_{\sigma}\frac{\sign(\sigma)}{l!2^l}x^{(n_1)}_{u_{\sigma(1)}u_{\sigma(2)}} x^{(n_2)}_{u_{\sigma(3)}u_{\sigma(4)}}\cdots x^{(n_l)}_{u_{\sigma(h-1)}u_{\sigma(h)}} 
	\end{equation}
The second summation is over all permutations $\sigma$ of $1,2,\cdots, h$ and $\sign(\sigma)$ is the sign of the permutation. 
\end{prop}

\subsection{Generators}  Recall that the Pfaffian $P(B)$ of the matrix $B$ in \ref{eqn:minor}, is represented by the ordered $h$-tuple $|u_h,\dots,u_2,u_1|$ with $1\leq u_i< u_{i+1}\leq p$ and $h\in 2\Zplus$. Similarly, let
\begin{equation}\label{eqn:sequenceJ}
J=\bpartial^n|u_h,\dots,u_2,u_1|
\end{equation} 
represent $\bpartial^nP(B)\in\R$, the $n^{\text{th}}$ normalized derivative of $P(B)$. For convenience, we use the notation $\bpartial^0|u_h,\dots,u_2,u_1|$ instead of $|u_h,\dots,u_2,u_1|$ when $n=0$, and we shall call such expressions {\it $\bpartial$-lists} throughout this paper. We call $wt(J)=n$  the weight of $J$ and call $sz(J)=h$ the size of $J$. Let $\cJ$ be the set of these $\bpartial$-lists,
and
$$\cJ_h=\{J\in \cJ| sz(J)\leq h\}$$ 
be the set of elements in $\cJ$ with size less than or equal to $h$.
Let $\cE$ be the set of ordered $h$-tuples of ordered pairs of the form
\begin{equation}\label{eqn:element}
E=|(u_h,k_h),\dots,(u_2,k_2),(u_1,k_1)|
\end{equation}
with $1\leq u_i\leq p$, $u_i\neq u_j$ if $i\neq j$ and $k_i\in \Zplus$.
Let
$$
||E||=\bpartial^n|u_{\sigma(h)},\dots,u_{\sigma(2)},u_{\sigma(1)}|\in\cJ.
$$
Here
$n=\sum k_i$, and $\sigma$ is the permutation of $1,2,\dots, h$ such that $u_{\sigma(i)}< u_{\sigma(i+1)}$. Let 
$$wt(E)=wt(||E||),\quad\quad sz(E)=sz(||E||).$$
Let $$\cE_h=\{E\in\cE|\,\,sz(E)\leq h\}.$$
 For $J\in \cJ$, let
$$\cE(J)=\{E\in\cE| \ ||E||=J\}.$$
Note that the Pfaffians represented by $\cJ$ form a set of generators of $\R$, and each $J \in \cJ$ can be represented by an element of the set $\cE(J)$.

\subsection{Ordering} For any set $\cS$, let $\cM(\cS)$ be the set of ordered products of elements of $\cS$.
If $\cS$ is an ordered set, we order the set  $\cM(\cS)$ lexicographically, that is 
$$
S_1S_2\cdots S_m\prec S'_1S'_2\cdots S'_n \quad\text{ if  } S_i=S'_i, i< i_0, \text{ with }S_{i_0}\prec S'_{i_0} \text { or } i_0=m+1,n>m.
$$ 
We order $\cM(\mathbb Z)$, the set of ordered products of integers, lexicographically.

There is an ordering on the set $ \cJ$:
$$
\bpartial^k|u_h,\dots,u_2,u_1|\prec \bpartial^{k'}|u'_{h'},\dots,u'_2,u'_1|
$$
if
\begin{itemize}
	\item
	$h'< h$ ;
	\item or $h'=h$ and $k< k'$;
	\item or  $h'=h$, $k= k'$  and $u_h\cdots u_1\prec u'_h\cdots   u'_1$. Here we order the words of natural numbers lexicographically.
\end{itemize}	
	
We order the pairs $(u,k)\in \Zplus\times \Zplus$ by
$$(u,k)\leq (u',k'), \text{ if } k<k' \text{ or } k=k' \text{ and }u\leq u'.$$
There is a partial ordering on the set $\cE$:
$$
|(u_h,k_h),\dots,(u_1,k_1)|\leq |(u'_{h'},k'_{h'}),\dots,(u'_1,k'_1)|
$$
if $h'\leq h$ and $(u_i,k_i)\leq (u'_i,k'_i)$, for $1\leq i\leq h'$.\\ 
Finally, there is an ordering on $\cE$:
$$ 
|(u_h,k_h),\dots,(u_1,k_1)|\prec |(u'_{h'},k'_{h'}),\dots,(u'_1,k'_1)|
$$
if
\begin{itemize}
	\item $h>h'$;
	\item or $h=h'$ and $\sum k_i<\sum k'_i$;
	\item or  $h=h'$, $\sum k_i=\sum k'_i$ and
	$$
	(u_h,k_h)\cdots(u_1,k_1)\prec (u'_{h'},k'_{h'})\cdots(u'_1,k'_1).
	$$
	Here we order the words of $\Zplus\times \Zplus$ lexicographically.
\end{itemize}
\begin{lemma}\label{lemma:compare2}
	If $E\leq E'$, then $||E||\prec ||E'||$.
\end{lemma}
\begin{proof}
	If $sz(E')<sz(E)$ or $sz(E)=sz(E')$ and $wt(E)<wt(E')$, then $||E||\prec ||E'||$.
	
	If $sz(E)=sz(E')$ and $wt(E)=wt(E')$, we must have $k_i=k_i'$ . So $u_i\leq u'_i$, we have $||E||\prec ||E'||$.
\end{proof}

\subsection{Relations}
In the previous notation of the $\bpartial$-list $\bpartial^k|u_h,\dots,u_1|$, we require $u_i<u_{i+1}$. To describe  their relations, we extend this notation $\bpartial^k|u_h,\dots,u_1|$ to any $1\leq u_i\leq p$. $\bpartial^k|u_h,\dots,u_1|$  still represents $\bpartial^k P(B)$ with $B$ in Equation (\ref{eqn:minor}) without the requirement $u_i<u_j$. Thus we have two obvious relations:
$$\bpartial^k|u_h,\dots,u_1|=0$$ if there is $1\leq i<j\leq h$ such that  $u_i=u_j$ and 
$$\bpartial^k|u_{\sigma(h)},\dots,u_{\sigma(2)},u_{\sigma(1)}|=\sign(\sigma)J$$
for a $\bpartial$-list $J\in \cJ$ of (\ref{eqn:sequenceJ}) and a permutation $\sigma$  of $1,2,\dots,h$.
\begin{lemma}\label{lem:basicrelation}
	For $0\leq k<s$, and $h' \geq s+1$, we have
	\begin{eqnarray}\label{eqn:basicrelation}
\sum_{\sigma} \frac{1}{h!s!}\sign(\sigma)	\bpartial^{0} |u_{\sigma(h+s)},\dots,u_{\sigma(s+1)}|\,\,
	\bpartial^{k} |u'_{h'},\dots,u'_{s+1},u_{\sigma(s)},\dots,u_{\sigma(1)}|
	\in \R[h+2].\nonumber
	\end{eqnarray}
	Here the summation is over all permutations $\sigma$ of $1,2,\dots,h+s$. \end{lemma}
\begin{proof}It is easy to see that 
	\begin{equation}\label{eqn:basicrelation0} \bpartial^{0}|u_h,\dots,u_1|=\sum_{i=2}^h(-1)^ix^{(0)}_{1,i}\bpartial^{0}|u_h,\dots,u_{i+1},u_{i-1},\dots,u_2|;
	\end{equation}

\begin{equation}
	\label{eqn:basicrelation1}
	\bpartial^{1}|u_h,\dots,u_1|=\sum_{1\leq i<j\leq h}(-1)^{i+j+1}x^{(1)}_{i,j}\bpartial^{0}|u_h,\dots,u_{j+1},u_{j-1},\dots,u_{i+1},u_{i-1},\dots,u_1|.
	\end{equation}
Let $l'=h'/2$. For $0\leq k<s$, $$\bpartial^{k} |u'_{h'},\dots,u'_{s+1},u_{\sigma(s)},\dots,u_{\sigma(1)}|=\sum_{k_1+\cdots+k_{l'}=k}
\sum_{\sigma'}\frac{\sign(\sigma')}{l'!2^{l'}} x^{(k_1)}_{\sigma'(u_{\sigma(1)})\sigma'(u_{\sigma(2)})}\cdots x^{(k_{l'})}_{\sigma'(u'_{h'-1})\sigma'(u'_{h'})}$$
\begin{equation}\label{eqn:pfexpression}
=\sum_{j=s+1}^{h'}\sum_{i=1}^s x^{(0)}_{u_{\sigma(i)} u'_j}f_{ij,\sigma}+\sum_{1\leq i<j\leq s} x^{(0)}_{u_{\sigma(i)} u_{\sigma(j)}}g_{ij,\sigma}+\sum_{1\leq i<j\leq s} x^{(1)}_{u_{\sigma(i)} u_{\sigma(j)}}h_{ij,\sigma}.
\end{equation}
Here $f_{ij,\sigma}, g_{ij,\sigma}, h_{ij,\sigma}\in\R$, and 
if $\sigma_1$ is a permutation of $1,2,\dots,s$ and $\sigma_2$ is a permutation of $s+1,\dots, h+s$, then
\begin{eqnarray*}x^{(0)}_{u_{\sigma(i)} u}f_{iu,\sigma}&=&\sign(\sigma_1) x^{(0)}_{u_{\sigma(i)} u}f_{\sigma_1^{-1}(i)u,\sigma\sigma_1\sigma_2},\\
 x^{(0)}_{u_{\sigma(i)} u_{\sigma(j)}}g_{ij,\sigma}&=&\sign(\sigma_1) x^{(0)}_{u_{\sigma(i)} u_{\sigma(j)}} g_{\sigma^{-1}_1(i)\sigma^{-1}_1(j),\sigma\sigma_1\sigma_2},\\  
x^{(1)}_{u_{\sigma(i)} u_{\sigma(j)}}h_{ij,\sigma}&=&\sign(\sigma_1) x^{(1)}_{u_{\sigma(i)} u_{\sigma(j)}} h_{\sigma^{-1}_1(i)\sigma^{-1}_1(j),\sigma\sigma_1\sigma_2}.
\end{eqnarray*}
We can also require that
if $\sigma_3$ is a permutation of $i,s+1,\dots,h+s$,
$f_{iu,\sigma}=  f_{iu  ,\sigma\sigma_3}$ and if $\sigma_4$ is a permutation of $i,j,s+1,\dots,h+s$,
$g_{ij,\sigma}=  g_{ij  ,\sigma\sigma_4}$, $h_{ij,\sigma}=  h_{ij  ,\sigma\sigma_4}$.
So by Equation (\ref{eqn:basicrelation0}), (\ref{eqn:basicrelation1}) and (\ref{eqn:pfexpression}),
\begin{eqnarray*}
&\sum_{\sigma} &\frac{1}{h!s!}\sign(\sigma)	\bpartial^{0} |u_{\sigma(h+s)},\dots,u_{\sigma(s+1)}|\,\,
\bpartial^{k} |u'_{h'},\dots,u'_{s+1},u_{\sigma(s)},\dots,u_{\sigma(1)}|
\\
&=&\sum_{\sigma}\frac{\sign(\sigma)}{h!s!}\sum_{j=s+1}^{h'}\sum_{i=1}^s\frac{-1}{h+1}	\bpartial^{0} |u_{\sigma(h+s)},\dots,u_{\sigma(s+1)},u_{\sigma(i)}, u'_j|\,f_{ij,\sigma}\\
&&+\sum_{\sigma}\frac{\sign(\sigma)}{h!s!}\sum_{1\leq i<j\leq s} \frac{1}{(h+1)}	\bpartial^{0} |u_{\sigma(h+s)},\dots,u_{\sigma(s+1)},u_{\sigma(j)}, u_{\sigma(i)}|\, g_{ij,\sigma}\\
&&+\sum_{\sigma}\frac{\sign(\sigma)}{h!s!}\sum_{1\leq i<j\leq s} \frac{1}{(h+1)(h+2)}	\bpartial^{1} |u_{\sigma(h+s)},\dots,u_{\sigma(s+1)},u_{\sigma(j)}, u_{\sigma(i)}|\,h_{ij,\sigma}\\
&=&\sum_{j=s+1}^{h'}\sum_{\substack{\sigma(h+s)>\cdots>\sigma(s)\\ \sigma(s-1)>\cdots>\sigma(1)}}-\sign(\sigma)	\bpartial^{0} |u_{\sigma(h+s)},\dots,u_{\sigma(s+1)},u_{\sigma(s)}, u'_j|\,f_{sj,\sigma}\\
&&+\sum_{\substack{\sigma(h+s)>\cdots>\sigma(s)\\ \sigma(s-2)>\cdots>\sigma(1)}}	\bpartial^{0} |u_{\sigma(h+s)},\dots,u_{\sigma(s+1)},u_{\sigma(s)}, u_{\sigma(s-1)}|\, g_{ij,\sigma}\\
&&+\sum_{\substack{\sigma(h+s)>\cdots>\sigma(s-1)\\ \sigma(s-2)>\cdots>\sigma(1)}}-\sign(\sigma)	\bpartial^{0} |u_{\sigma(h+s)},\dots,u_{\sigma(s+1)},u_{\sigma(s)}, u_{\sigma(s-1)}|\, h_{ij,\sigma}\\
&&\in \R[h+2]
\end{eqnarray*}
\end{proof}

\begin{lemma}\label{lem:fullrelation0}
   	For $i,j,h,h',k_0,m\in \mathbb Z_{\geq 0}$ with $h\geq h'$, $i\leq h$, $j\leq h'$ and $k_0\leq m$, let $l_0=i+j-h-1$.  Given any integers $a_{k}$ for $k_0\leq k\leq k_0+l_0$, there are integers $a_k$ in the range $0\leq k<k_0$, and $k_0+l_0<k\leq m$, such that 
	\begin{eqnarray}\label{eqn: relation4'}
	&\sum_{k=0}^ma_k\sum_{\sigma} \frac{1}{i!j!}\sign(\sigma)\quad\quad\quad\quad\quad\quad\quad\quad\quad\quad &\\
	&\bpartial^{m-k} |u_h,\dots,u_{i+1},\sigma(u_{i}),\dots,\sigma(u_1)|\,\,
	\bpartial^{k} |u'_{h'},\dots,u'_{j+1},\sigma(u'_{j}),\dots,\sigma(u'_1)| &
	\in \R[h+2].\nonumber
	\end{eqnarray}
	Here the second summation is over all permutations $\sigma$ of $u_{i},\dots,u_1,u'_{j},\dots,u'_1$ and 
	$\sign(\sigma)$ is the sign of the permutation.
\end{lemma}

For simplicity, we write Equation (\ref{eqn: relation4'}) in the following way,
	\begin{equation}\label{eqn: relation4}
\sum \epsilon a_{k} 
\bpartial^{m-k} |u_h,\dots,u_{i+1},\underline{u_{i},\dots,u_1}|
\bpartial^{k} |u'_{h'},\dots,u'_{j+1},\underline{u'_{j},\dots,u'_1}|
\in \R[h+2].
\end{equation}
\begin{remark}
Since the second summation in Equation (\ref{eqn: relation4'}) is over all permutations, each monomial in the equation will appear $i!j!$ times, and the coefficient of each monomial will be $\pm a_k$.	
\end{remark}
\begin{proof}[Proof of Lemma \ref{lem:fullrelation0}]
	Let $\cF_l(u_h,\dots,u_{i+1}; u'_{h'},\dots,u'_{j+1})$
	$$
	=
	\sum_{\sigma} \frac{\sign(\sigma)}{i!j!} 
	\bpartial^{0} |u_h,\dots,u_{i+1},\sigma(u_{i}),\dots,\sigma(u_1)|\,\,
	\bpartial^{l} |u'_{h'},\dots,u'_{j+1},\sigma(u'_{j}),\dots,\sigma(u'_1)|.
	$$
	We have
	$\cF_l(u_h,\dots,u_{i}; u'_{h'},\dots,u'_{j+1})$
	$$=\cF_l(u_h,\dots,u_{i+1};  u'_{h'},\dots,u'_{j+1})\pm\cF_l(u_h,\dots,u_{i+1}; u'_{h'},\dots,u'_{j+1},u_{i}).$$
	If $i=h$ and $0\leq l\leq l_0$, by Lemma (\ref{lem:basicrelation}), $\cF_l(\,\,; u'_{h'},\dots,u'_{j+1})\in\R[h+2]$.
	By induction on $h-i$, we can see that $\cF_l(u_h,\dots,u_{i+1}; u'_{h'},\dots,u'_{j+1})\in\R[h+2].$
	Thus 
		$$\sum_{k=0}^m C_k^l\sum_{\sigma} \frac{\sign(\sigma)}{i!j!} 
	\bpartial^{m-k} |u_h,\dots,u_{i+1},\sigma(u_{i}),\dots,\sigma(u_1)|\,\,
	\bpartial^{k} |u'_{h'},\dots,u'_{j+1},\sigma(u'_{j}),\dots,\sigma(u'_1)
	$$
	$$=\bpartial^{m-l}\cF_l(u_h,\dots,u_{i+1}; u'_{h'},\dots,u'_{j+1})\in \R[h+2].$$
	Now the $(l_0+1)\times (l_0+1)$ integer matrix with entries $c_{ji}=C_{k_0+j}^i$, $0\leq i,j\leq l_0$ is invertible since the determinant of this matrix is $\pm1$. Let $b_{ij} \in \mathbb{Z}$ be the entries of the inverse matrix. Let $a_{k}= \sum_{l=0}^{l_0}\sum_{j=0}^{l_0}C_{k}^l b_{l,j}a_{k_0+j}$. 
	So 
	$$\text{left hand side of Equation (\ref{eqn: relation4'})}=\sum_{l=0}^{l_0}\sum_{j=0}^{l_0} b_{l,j}a_{k_0+j}\bpartial^{m-l}\cF_l(u_h,\dots,u_{i+1}; u'_{h'},\dots,u'_{j+1})\in \R[h+2].
	$$
\end{proof}

\subsection{Standard monomials} Now we give the definition of standard monomials of $\cJ$.
\begin{defn}\label{def:standard}
	An ordered product $E_{1}E_{2}\cdots E_{m}$ of elements of $\cE$ is said to be standard if \begin{enumerate}
	\item $E_a\leq E_{a+1}$, $1\leq a<m$, 
	\item  $E_{1}$ is the largest in $\cE(||E_{1}||)$ under the order $\prec$,
	\item $E_{a+1}$ is the largest in $\cE(||E_{a+1}||)$ such that $E_a\leq E_{a+1}$. 
	\end{enumerate}
	An ordered product $J_1J_2\cdots J_m$ of elements of $\cJ$ is said to be standard
	if there is a standard ordered product $E_{1}E_{2}\cdots E_{m}$ such that $E_i\in \cE(J_i)$.
\end{defn}	
Let $\cS\cM(\cJ)\subset \cM(\cJ)$ be the set of standard monomials of $\cJ$. 
Let $\cS\cM(\cE)\subset \cM(\cE)$ be the set of standard monomials of $\cE$. Let $\cS\cM(\cJ_h)=\cM(\cJ_h)\cap\cS\cM(\cJ)$ be the set of standard monomials of $\cJ_h$. Let $\cS\cM(\cE_h)=\cM(\cE_h)\cap\cS\cM(\cE)$ be the set of standard monomials of $\cE_h$.

By Definition \ref{def:standard}, if $J_1J_2\cdots J_m$ is standard, the standard monomial $E_1\cdots E_m\in \cS\cM(\cE)$ corresponding to $J_1\cdots J_m$ is unique and $E_1$ has the form 
$$|(u_h,wt(E_1)),(u_{h-1},0),\dots,(u_1,0)|\in \cE$$
 with $u_i<u_{i+1}$. Therefore the map
 $$\pi_h :\cS\cM(\cE_h) \to \cS\cM(\cJ_h  ),\quad E_1E_2\cdots E_m\mapsto ||E_1|| ||E_2|| \cdots ||E_m||$$
is a bijection.

We order $\cM(\cJ)$, the set of ordered products of elements of $\cJ$, lexicographically. The following lemma will be proved later in Section \ref{section:proof2.5}.
\begin{lemma}\label{lem:base0}
	If $J = J_1\cdots J_b\in\cM(\cJ)$ is not standard, $J$ can be written as a linear combination of elements of $\cM(\cJ)$ preceding $J_1\cdots J_{b-1}$, with integer coefficients.
\end{lemma}

Recall that $\R[h]$ denotes the ideal generated by $J\in \cJ$ with $sz(J)=h$, and $\R_h=\R\slash\R[h+2]$. If $h\geq p$, then $\cJ_h=\cJ$ and $\R_h=\R$.
By the above lemma, we immediately have the following lemma.
\begin{lemma}\label{lem:base1}
	Any element of $\R_h$ can be written as a linear combination of standard monomials of $\cJ_h$ with integer coefficients.
\end{lemma}
\begin{proof}
We only need to show that any element of $\R$ can be written as a linear combination of standard monomials of $\cJ$ with integer coefficients. If the lemma is not true,
there must be a smallest element $J\in \cM(\cJ)$, which cannot be written as a linear combination of elements of $\cS\cM(\cJ)$ with integer coefficients. So $J$ is not standard. By Lemma \ref{lem:base0},
$J=\sum_{\alpha}\pm  J_{\alpha}$ with $J_{\alpha}\in \cM(\cJ)$ and $J_{\alpha}\prec J$. Since $J_{\alpha}$ can be written as a linear combination of elements of $\cS\cM(\cJ)$ with integer coefficients, 
$J$ can also be written as such a linear combination, which is a contradiction.
\end{proof}

\section{A canonical basis} \label{section:3}
\subsection{A ring homomorphism}
Let  
$$
\cS_h=\{a^{(k)}_{il}|\ 1\leq i\leq p, \ 1\leq l\leq h,\  k\in \Zplus \},
$$
and let \begin{equation} \label{defofB} \B=\mathbb Z [\cS_h],\end{equation} the polynomial ring generated by $\cS_h$. Note that for a field $K$, if $W = K^{\oplus h}$ and $V = W^{\oplus p}$, the affine coordinate ring $K[J_{\infty}(V)]$ is obtained from $\B$ by base change, i.e., $K[J_{\infty}(V)] \cong \B \otimes_{\mathbb{Z}} K$.

Let $\partial$ be the derivation on $\B$ given by
$\partial a^{(k)}_{ij}=(k+1)a^{(k+1)}_{ij}$, and let $\bpartial = \frac{1}{k!} \partial$ as before. We have a ring homomorphism
$$\tilde Q_h: \R\to \B,\quad\quad  x^{(k)}_{uv}\mapsto \bpartial^k\sum_{i=1}^{h/2} (a^{(0)}_{u{2i-1}}a^{(0)}_{v2i}-a^{(0)}_{v{2i-1}}a^{(0)}_{u2i}).$$
For any $J\in \cJ$ with $sz(J)>h$, we have $\tilde Q_h(J)=0$, so $\tilde Q_h$ induces a ring homomorphism 
\begin{equation} \label{def:qh} Q_h:\R_h\to \B.\end{equation}

\subsection{Tableaux}
Let $\tilde \cS_h$=$\cS_h\cup\{*\}$.
We define an ordering on the set $\tilde\cS_h$:\\ 
for $a^{(k)}_{ij}, a^{(k')}_{i'j'}\in \cS_h$,
$a^{(k)}_{ij}<*$ 
and
$a^{(k)}_{ij}\leq a^{(k')}_{i'j'}$
 if $k i j\prec k' i' j'$.

We use tableaux to represent the monomials of $\B$. Let $\cT$ be the set of the following tableaux:
\begin{equation}\label{eqn:table}
\left|\begin{array}{c}
y_{1,h},\cdots,y_{1,2}, y_{1,1} \\
\vdots  \\
y_{m,h},\cdots,y_{m,2}, y_{m,1}  \\
\end{array}\right|.
\end{equation}
Here $y_{s,l}$ are some $a^{(k)}_{il}$ or  $*$, every row of the tableau has elements in $\cS_h$ and
 $y_{s,j}\leq y_{s+1,j}.$
 We use the tableau (\ref{eqn:table}) to represent a monomial in $\B$, which is the product of  ${a^{(k)}_{ij}}'s$ in the tableau.
It is easy to see that the representation is a one-to-one correspondence between $\cT$ and the set of monomials of $\B$.
We associate to the tableau (\ref{eqn:table}) the word:
$$ 
 y_{1,h}\cdots y_{1,1}  y_{2,h}\cdots  y_{2,1} \cdots y_{m,h}\cdots y_{m,1}
 $$
and order these words lexicographically.
For a polynomial $f\in \B$, let $Ld(f)$ be its leading monomial in $f$ under the order we defined on $\cT$.

For $E_i=|(u^i_{h_1},k^i_{h_1}),\dots,(u^i_2,k^i_2),(u^i_1,k^i_1)|\in \cE$, $1\leq i\leq m$,
we use a tableau to represent $E_1\cdots E_m\in\cS\cM(\cE)$,
\begin{equation}\label{eqn:tableE} \left|
\begin{array}{c}
(u^1_{h_1},k^1_{h_1}),\cdots,(u^1_2,k^1_2),(u^1_1,k^1_1)\\
(u^2_{h_2},k^2_{h_2}),\cdots,(u^2_2,k^2_2),(u^2_1,k^2_1)\\
\vdots \\
(u^m_{h_m},k^m_{h_m}),\cdots,(u^m_2,k^m_2),(u^m_1,k^m_1)\\
\end{array}
\right|.
\end{equation}
Let $T:\cS\cM(\cJ_h)\to \cT$ with
$$ T(E_1\cdots E_m)=\left|
\begin{array}{c}
*,\cdots, *, a_{u^1_{h_1}h_1}^{(k^1_{h_1})},\cdots,a_{u^1_{1}1}^{(k^1_{1})}\\
*,\cdots, *, a_{u^2_{h_2}h_2}^{(k^2_{h_2})},\cdots,a_{u^2_{1}1}^{(k^2_{1})}\\
\vdots \\
*,\cdots, *, a_{u^m_{h_m}h_m}^{(k^m_{h_m})},\cdots,a_{u^m_{1}1}^{(k^m_{1})}
\end{array}
\right|.
$$
Obviously, $T$ is an injective map and $T(E)\prec T(E')$ if $E\prec E'$.
\begin{lemma}\label{lem:leadingterm}
	Let $J_1\cdots J_m\in \cS\cM(\cJ_h)$ and $E_1\cdots E_m\in \cS\cM(\cE_h)$ be its associated standard monomial. Assume the tableau representing $E_1\cdots E_m$ is (\ref{eqn:tableE}). Then
	the leading monomial of $Q_h(J_1\cdots J_m)$ is represented by the tableau $T(E_1E_2\cdots E_m)$.  
	Thus $$Ld\circ Q_h=T\circ \pi_h^{-1}: \cS\cM(\cJ_h)\to \cT$$  is injective. The coefficient of the leading monomial of $Q_h(J_1\cdots J_m)$ is $\pm 1$.
\end{lemma}
\begin{proof}
	Let $W_m$ be the monomial corresponding to the tableau $T(E_1\cdots E_m)$. Let $$
	M_m=a_{u^m_{h_m}h_m}^{(k^m_{h_m})}\cdots a_{u^m_{1}1}^{(k^m_{1})}
	$$
	 be the monomial corresponding to the tableau $T(E_m)$.
	Then $W_m=W_{m-1}M_m$. Let $l_m=\frac{h_m}{2}$, by a direct calculation,	
	$$
	Q_h(J_m)	=\sum \pm a_{u^m_{\sigma(1)} 2s_1-1}^{(k_1)}a_{u^m_{\sigma(2)} 2s_1}^{(k_2)}a_{u^m_{\sigma(3)} 2s_2-1}^{(k_3)}a_{u^m_{\sigma(4)} 2s_2}^{(k_4)}\cdots a_{u^{m}_{\sigma(h_{m}-1)}2s_{l_m}-1}^{(k_{h_m-1})}a_{u^{m}_{\sigma(h_{m})}2s_{l_m}}^{(k_{h_m})}.
	$$
	The summation is over all $k_i\geq 0$ with $\sum k_i=wt(E_m)$, all $s_i$ with $1\leq s_1< s_2<\cdots<s_{l_m}\leq h/2$ and
	all permutations $\sigma$ of  $1,2,\dots,h_m$.

	We prove the lemma by induction on $m$. If $m=1$, the lemma is obviously true. Assume it is true for $J_1\cdots J_{m-1}$. Then $Ld(Q_h(J_1\cdots J_{m-1}))=W_{m-1}$, the monomial corresponding to $T(E_1\cdots E_{m-1})$, and the coefficient of $W_{m-1}$ in $Q_h(J_1\cdots J_{m-1})$ is $\pm 1$.
$M_m$ is one of the monomials in $Q_h(J_m)$ with coefficient $\pm 1$.
	 All of the monomials in $Q_h(J_1\cdots J_{m-1})$  except $W_{n-1}$ are less than
	 $W_{n-1}$, so any monomial in $Q_h(J_1\cdots J_{m-1})$ except $W_{n-1}$ times any monomial in $Q_h(J_m)$  is less than $W_{m-1}$. Since $W_{m-1}\prec W_m$, the coefficient of $W_m$ in $Q_h(J_1\cdots J_{m})$ is 
	 not zero.  Now $$W_{m-1}\prec W_m\prec Ld(Q_h(J_1\cdots J_{m})).$$
The leading monomial $Ld(Q_h(J_1\cdots J_{m}))$ must have the form
	$$W=W_{m-1}a_{u^m_{\sigma(1)} 2s_1-1}^{(k_1)}a_{u^m_{\sigma(2)} 2s_1}^{(k_2)}a_{u^m_{\sigma(3)} 2s_2-1}^{(k_3)}a_{u^m_{\sigma(4)} 2s_2}^{(k_4)}\cdots a_{u^{m}_{\sigma(h_{m}-1)}2s_{l_m}-1}^{(k_{h_m-1})}a_{u^{m}_{\sigma(h_{m})}2s_{l_m}}^{(k_{h_m})}.$$
	If some $s_i $ is greater than $h_{m-1}/2$, then $W\prec W_{n-1}$. 	So $s_i\leq h_{m-1}/2$. We must have $a_{u^m_{i}s_i}^{(k_{i})}\geq a_{u^{m-1}_{s_i}s_i}^{(k^{m-1}_{s_i})}$; otherwise, $W\prec W_{m-1}$.
	 If there is some $h_{m-1}/2\geq s_i>l_m$, then $W\prec W_m$. So we can assume $s_i=i$. 
Such monomials in 
$Q_h(J_m)$ are in one-to-one correspondence with $E'_m\in\cE(J_m)$ such that $E_{m-1}\leq  E'_m$.  $E_{m}$ is the largest in $\cE(J_m)$ with $E_{m-1}\leq E_m$ since $E$ is standard, so $W_m$ is the leading term of $Q_h(J_1\cdots J_m)$. The coefficient of $W_m$ in $Q_h(J_1\cdots J_m)$
is $\pm 1$ since the coefficients of $W_{m-1}$ in $Q_h(J_1\cdots J_{m-1})$ and $M_m$ in $Q_h(J_m)$ are $\pm 1$. 
\end{proof}

\begin{proof}[Proof of Theorem \ref{thm:standard}.]
 By Lemma \ref{lem:leadingterm}, $Ld(Q_h(\cS\cM(\cJ_h)))$ are linearly independent, so $\cS\cM(\cJ_h)$ are linearly independent. By Lemma \ref{lem:base1}, $\cS\cM(\cJ_h)$ generates $\R_h$. So  $\cS\cM(\cJ_h)$ is a $\mathbb Z$-basis of $\R_h$.
\end{proof}

\begin{thm}\label{thm:standmonomial1} 
 $Q_h:\R_h\to \B$ is injective. So we may identify $\R_h$ with its image $\text{Im}(Q_h)$, which is the subring of $\B$ generated by $\bpartial^k\sum_i^{h/2} (a^{(0)}_{u{2i-1}}a^{(0)}_{v2i}-a^{(0)}_{v{2i-1}}a^{(0)}_{u2i})$. In particular, $Q_h(\cS\cM(\cJ_h))$ is a $\mathbb Z$-basis of $\text{Im}(Q_h)$.
\end{thm}

\begin{proof}
By Lemma \ref{lem:leadingterm}, $Ld(Q_h(\cS\cM(\cJ_h)))$ are linearly independent. Since $\cS\cM(\cJ_h)$ is a $\mathbb Z$-basis of $\R_h$, $Q_h:\R_h\to \B$ is injective.   
\end{proof}
Since $Q_h$ is injective and $\B$ is an integral domain, we obtain the following corollary.
\begin{cor}
$\R_h$ is an integral domain.
\end{cor}

\section{Application} \label{section:4}
In this section, we give the main application of our standard monomial basis, which is the arc space analogue of Theorem \ref{classicalFFTSFT}.

\subsection{Arc spaces} Suppose that $X$ is a scheme of finite type over $K$. Its arc space (cf.~\cite{EM}) $J_\infty(X)$ is determined by its functor of points. For every $K$-algebra $A$, we have a bijection
$$\Hom(\Spec A, J_\infty(X))\cong\Hom(\Spec A[[t]],  X).$$
If $i: X\to Y$ is a morphism of schemes, we get a morphism of schemes $i_{\infty}:J_{\infty}(X)\to J_{\infty}(Y)$.
If $i$ is a closed immersion, then $i_{\infty}$ is a closed immersion.
If $i:X\to Y$ is an \'etale morphism, 
 then we have a Cartesian diagram
 $$\begin{array}{ccc}
 	J_\infty(X)& \to  & J_\infty(Y) \\ 
 	\downarrow&  &\downarrow  \\ 
 	X & \to  & Y 
 \end{array} 
$$

 If $X=\Spec K[x_1,\dots,x_n]$, then $J_{\infty}(Y)=\Spec K[x^{(k)}_i|1\leq i\leq n, k\in \Zplus]$. The identification is made as follows: for a $K$-algebra $A$, a morphism $\phi: K[x_1,\dots, x_n]\to A[[t]]$ determined by $\phi(x_i)=\sum_{k=0}^\infty a_i^{(k)}t^k$ corresponds to a morphism
$K[x_i^{(k)}]\to A$ determined by $x_i^{(k)}\to a_i^{(k)}$. Note that $K[x_1,\dots,x_n]$ can be identified with the subalgebra $K[x^{(0)}_1,\dots,x^{(0)}_n] \subset K[x_i^{(k)}]$, and from now on we use $x_i^{(0)}$ instead of $x_i$.

The polynomial ring $K[x^{(k)}_i]$ has a derivation $\partial$ defined on generators by 
\begin{equation} \label{def:partial} \partial x_i^{(k)}=(k+1)x_i^{(k+1)}.\end{equation}
It is more convenient to work with the normalized $k$-derivation $\bpartial^k = \frac 1{k!}\partial^k$, but this is a priori not well-defined on $K[x_i^{(k)}]$ if $\text{char}\ K$ is positive. But $\partial$ is well-defined on $\mathbb Z[x_i^{(k)}]$, and $\bpartial^k$ maps $\mathbb Z[x_i^{(k)}]$ to itself, so for any $K$, there is an induced $K$-linear map 
\begin{equation} \label{def:bpartial} \bpartial^k:K[x_i^{(k)}]\to K[x_i^{(k)}],\end{equation} obtained by tensoring with $K$.

If $X$ is the affine scheme $\text{Spec} \ K[x^{(0)}_1,\dots, x^{(0)}_n]/(f_1,\dots, f_r)$, then $J_\infty(X)$ is the affine scheme
$$\text{Spec}\ K[x_i^{(k)}|\ i =1,\dots, \ k \in \Zplus]/ ( \bpartial^l f_j| \ j = 1,\dots, r,\ l \geq 0 ).$$ Indeed, for every $f\in K[x^{(0)}_1,\dots, x^{(0)}_n]$, we have 
$$\phi(f)=\sum_{k=0}^\infty (\bpartial^kf)(a_1^{(0)},\dots, a_n^{(k)}) \,t^k.$$ It follows that $\phi$ induces a morphism $K[x^{(0)}_1,\dots, x^{(0)}_n]/(f_1,\dots, f_r)\to A[[t]]$ if and only if 
$$(\bpartial^kf_i)(a_1^{(0)},\dots, a_n^{(k)})=0,\ \text{for all} \ i = 1,\dots, r,\ k\geq 0.$$
If $Y$ is the affine scheme $\text{Spec} \ K[y^{(0)}_1,\dots,y^{(0)}_m]/(g_1,\dots,g_s)$, a morphism $P:X\to Y$ gives a ring homomorphism $P^*:K[Y]\to K[X]$.
Then the induced homomorphism of arc spaces $P_\infty:J_\infty(X)\to J_\infty(Y)$ is given by
$$ P^*_\infty(y_i^{(k)})=\bpartial^k P^*(y_i^{(0)}).$$
In particular, $P^*_\infty$ commutes with $\bpartial^k$ for all $k\geq 0$.

\subsection{Arc space of the Pfaffian variety} Recall that the space $SM_p$ of skew-symmetric $p\times p$ matrices over $K$ has affine coordinate ring
$$K[SM_p] = K[x^{(0)}_{ij}|\ 1\leq i,\ j\leq p]/(x^{(0)}_{ij} + x^{(0)}_{ji}, x^{(0)}_{ii}) = R\otimes_{\mathbb{Z}} K,$$ where $R$ is given by \eqref{def:R}. The Pfaffian variety $Pf_h$ is the subvariety of $SM_{p}$ determined by the ideal $K[SM_p][h]$ generated by the Pfaffians of all diagonal $h$-minors, so 
$$K[Pf_h] = K[SM_p]/K[SM_p][h] = R_{h-2} \otimes_{\mathbb{Z}} K,$$ where $R_{h-2}$ is given by \eqref{def:Rh}. Similarly, recall that
$$K[J_{\infty}(SM_p)] = K[x^{(k)}_{ij}|\ 1\leq i,\ j\leq p]/(x^{(k)}_{ij}+x^{(k)}_{ji}, x^{(0)}_{ii}) = \R \otimes_{\mathbb{Z}} K.$$  Then
$$K[J_{\infty}(Pf_h)] =  K[J_{\infty}(SM_p)] / K[J_{\infty}(SM_p][h-2] \cong \R_{h-2} \otimes_{\mathbb{Z}} K.$$

\begin{proof}[Proof of Corollary \ref{cor:standarddeterm}]
	By Theorem \ref{thm:standard}, $\cS\cM(\cJ_{h-2})$ is a $\mathbb Z$-basis of $\R_{h-2}$. So it is a $K$-basis of $K[J_{\infty}(Pf_h)]$.
\end{proof}
Recall the map $Q_h: \R_h \rightarrow \B$ given by \eqref{def:qh}, which extends to a map
\begin{equation} \label{def:qhK} Q^K_h: K[J_{\infty}(Pf_{h+2})] \rightarrow K[J_{\infty}(V)],\end{equation} where $K[J_{\infty}(Pf_{h+2})]$ and $K[J_{\infty}(V)]$ are identified with $\R_h \otimes_{\mathbb{Z}} K$ and $\B \otimes_{\mathbb{Z}} K$, respectively, and $Q^K_h = Q_h \otimes \text{Id}$.

\begin{thm}\label{thm:injectiveGL}
	\label{thm:determintegral}
	$Q^K_h$ is injective, so we may identify $K[J_{\infty}(Pf_{h+2})]$ with the subring $\text{Im}(Q^K_h)$ of $K[J_{\infty}(V)]$. In particular, $K[J_{\infty}(Pf_{h+2})]$ is integral.
\end{thm}
\begin{proof}
	By Lemma \ref{lem:leadingterm}, $Ld(Q_h(\cS\cM(\cJ_h)))$ are linearly independent. By Corollary \ref{cor:standarddeterm}, $\cS\cM(\cJ_h)$ is a $K$-basis of $\R_h$,  so $Q^K_h$ is injective. Since  $K[J_{\infty}(V)]$ is integral, so is $K[J_{\infty}(Pf_{h+2})]$.
\end{proof}
In general, if $\text{char}\  K=0$, the arc space of an integral scheme is irreducible \cite{Kol} but it may not be reduced. $Pf_h(K)$ is an example whose arc space is integral.

\subsection{Principal $G$-bundles} Let $G$ be an algebraic group over $K$. If $G$ acts morphically on an algebraic variety $X$, then we say that $X$ is a $G$-variety. An affine $G$-variety $X$ is a principal
$G$-bundle in the the \'etale topology if, for every $x\in X/\!\!/G$, there is an \'etale neighborhood  $V\to X$ of $x$ such that $V\times_{X/G}X\cong V\times G$ as $G$-varieties.
The following proposition is from \cite{BR} by P. Bardsley and R. W. Richardson.
\begin{prop}\label{prop:principlebundle}
	Let $X$ be an affine $G$-variety. Then $X$ is a principal $G$-bundle in the \'etale topology if and only if for every $x$ in $X$, the orbit $G \cdot x$ is separable and the stabilizer $G_x$ is trivial.
\end{prop}

The group structure $G\times G\to G$ induces the group structure on its arc space
$$J_\infty(G)\times J_\infty(G)\to J_\infty(G).$$ 
So $J_\infty(G)$ is an algebraic group. For a $G$-variety $X$, the action $G\times X\to X$
induces the action of $J_\infty(G)$ on
$J_\infty (X)$,
$$J_\infty(G)\times J_\infty(X)\to J_\infty(X).$$
The quotient map $X\to X/\!\!/G$ induces morphisms $J_\infty(X)\to J_\infty(X/\!\!/G)$ and $$ \pi_X:J_\infty(X)/\!\!/J_\infty(G)\to J_\infty(X/\!\!/G).$$
\begin{prop}\label{prop:principlequotient}
	If $X$ is a principal $G$-bundle in the \'etale topology and $X/\!\!/G$ is smooth, then $ \pi_X$ is an isomorphism.
\end{prop}
\begin{proof} For any \'etale morphism $V\to X/\!\!/G$ with $V\times_{X/\!\!/G}X\cong V\times G$ as $G$-varieties, we have Cartesian diagrams
	$$\begin{array}{ccc}
	J_\infty(V)& \to  & J_\infty(X/\!\!/G) \\ 
	\downarrow&  &\downarrow  \\ 
	V & \to  & X/\!\!/G
	\end{array} \quad\quad \text{and}\quad\quad
	\begin{array}{ccc}
		J_\infty(V)\times J_\infty(G) & \to  & J_\infty(X) \\ 
		\downarrow&  &\downarrow  \\ 
		V\times G & \to & X\\
		\downarrow&  &\downarrow  \\ 
		V & \to  & X/\!\!/G
	\end{array} .	
	$$
	So $J_\infty(V) \to  J_\infty(X/\!\!/G)$ is an \'etale morphism and
	\begin{eqnarray*}J_\infty(V)\times_{J_\infty(X/\!\!/G)}J_\infty(X)/\!\!/J_\infty(G)&\cong & (J_\infty(V)\times_{J_\infty(X/\!\!/G)}J_\infty(X))/\!\!/J_\infty(G)\\
		&\cong &( V\times_{X/\!\!/G}J_\infty(X))/\!\!/J_\infty(G) \\
		&\cong & J_\infty(V)\times J_\infty(G)/\!\!/J_\infty(G)\\
		&\cong & J_\infty(V).
	\end{eqnarray*}
	If $\pi_X$ is not an isomorphism, 
there is an \'etale morphism $V\to X/\!\!/G$ such that $V\times_{X/\!\!/G}X\cong V\times G$ as $G$-varieties and $ \pi_V: J_\infty(V)\times_{J_\infty(X/\!\!/G)}J_\infty(X)/\!\!/J_\infty(G)\to J_\infty(V)$ is not an isomorphism. But $V\times _{X/\!\!/G}J_\infty(X/\!\!/G)\cong J_\infty(V)$, which is a contradiction. \end{proof}

\subsection{Invariants for the arc space of the symplectic group action}
Let $G=Sp_h(K)$ be the symplectic group over $K$, $W=K^{\oplus h}$ its standard representation, and $V=W^{\oplus p}$. Recall that $V$ has affine coordinate ring 
$$K[V] =K[a^{(0)}_{il}|\ 0\leq i,j\leq p,1\leq l\leq h].$$ The action of $G$ on $V$ induces an action of $J_{\infty}(G)$ on the affine coordinate ring 
$$K[J_{\infty}(V)] = K[a^{(k)}_{il}|\ 0\leq i,j\leq p,\ 1\leq l\leq h, \ k\in \mathbb{Z}_{\geq 0}],$$ which is identified with $\B \otimes_{\mathbb{Z}} K$, where $\B$ is given by \eqref{defofB}.

If $p\geq h$, 
let $\Delta=Q^K_{h}(|h, \dots, 1|)$, and let $K[J_{\infty}(V)]_{\Delta}$ and  $\text{Im}(Q^K_h)_{\Delta}$ be the localization of $K[J_{\infty}(V)]$ and $\text{Im}(Q^K_h)$ at $\Delta$.

\begin{lemma}\label{lemma:equalGLinvariant}If $p\geq h$,
	$$K[J_{\infty}(V)]_{\Delta}^{J_\infty({Sp}_h(K))} =\text{Im}(Q^K_h)_{\Delta}.$$ 
\end{lemma}
\begin{proof} 
	Let $K[V]_\Delta$ be the localization of $K[V]$ at $\Delta$ and $V_\Delta=\Spec K[V]_\Delta$. By Theorem \ref{classicalFFTSFT}, the ring of invariants $K[V]^{Sp_h}$ is generated by $Q_h^K(x^{(0)}_{uv})$, so the affine coordinate ring of $J_\infty(V_\Delta/\!\!/G)$ is isomorphic to $\text{Im}(Q^K_h)_{\Delta}$. To prove the lemma, we only need to show that $\pi_{V_\Delta}:J_\infty(V_\Delta)/\!\!/J_\infty(G)\cong J_\infty(V_\Delta/\!\!/G)$. By Proposition \ref{prop:principlebundle}, $V_\Delta$ is a principal bundle in the \'etale topology, and by Proposition \ref{prop:principlequotient}, $\pi_{V_\Delta}$ is an isomorphism.
\end{proof}
\begin{thm}\label{thm:JGLinvariant}
	$K[J_{\infty}(V)]^{J_\infty({Sp}_h(K))} =\text{Im}(Q^K_h)$. 
\end{thm}
\begin{proof} If $p\geq h$, we regard $K[J_{\infty}(V)]$ and $\text{Im}(Q^K_h)_{\Delta}$ as subrings of $K[J_{\infty}(V)]_{\Delta}$.
	By Lemma \ref{lemma:equalGLinvariant}, we have  
	$$K[J_{\infty}(V)]^{J_\infty({Sp}_h(K))}=K[J_{\infty}(V)] \cap \text{Im}(Q^K_h)_{\Delta}.$$ Now for any $f \in K[J_{\infty}(V)] \cap \text{Im}(Q^K_h)_{\Delta}$,
	$f=\frac{g}{\Delta^n}$, with $\Delta^nf=g\in \text{Im}(Q^K_h)$. The leading monomial of $g$ is $$Ld(g)=(a^{(0)}_{11}\cdots a^{(0)}_{hh})^n Ld(f),$$
	with coefficient $C_0\neq 0$.
	Since $g\in \text{Im}(Q^K_h)$, there is a standard monomial $J\in\cS\cM(\cJ_h)$, with $Ld(Q_h(J))=Ld(g)$. Since $J$ has the factor $|h, \dots, 1|^n$, $Q_h(J)$ has the factor $\Delta^n$.
	Thus $f-C_0\frac{Q^K_h(J)}{{\Delta}^n}\in K[J_{\infty}(V)] \cap \text{Im}(Q^K_h)_{\Delta}$ with a lower leading monomial and $\frac{Q^K_h(J)}{{\Delta}^n}\in \text{Im}(Q^K_h)$. By induction on the leading monomial of $f$, $f\in \text{Im}(Q^K_h)$, so $K[J_{\infty}(V)] \cap \text{Im}(Q^K_h)_{\Delta} =\text{Im}(Q^K_h)$, and  $K[J_{\infty}(V)]^{J_\infty({Sp}_h(K))}  =\text{Im}(Q^K_h)$.
	
	More generally, let $V' = W^{\oplus p+h}$ where $W = K^{\oplus h}$ as before. Its arc space has affine coordinate ring
	$$K[J_{\infty}(V')] = K[a^{(k)}_{il}|\ 0\leq i,j\leq p +h,\ 1\leq l\leq h, \ k\in \mathbb{Z}_{\geq 0}],$$ which contains $K[J_{\infty}(V)]$ as a subalgebra, and has an action of $J_{\infty}(G)$. By the above argument, $K[J_{\infty}(V')]^{J_{\infty}(G)}$ is generated by $X^{(k)}_{uv}=\bpartial^k\sum_{i=1}^{h/2} (a^{(0)}_{u{2i-1}}a^{(0)}_{v2i}-a^{(0)}_{v{2i-1}}a^{(0)}_{u2i})$.
Let $\cI$ be the ideal of $K[J_{\infty}(V')]$ generated by $a^{(k)}_{il}$ with $i>p$. Then 
$$K[J_{\infty}(V')] = K[J_{\infty}(V)] \oplus \cI.$$ Note that $K[J_{\infty}(V)]$ and $\cI$ are $J_{\infty}(G)$-invariant subspaces of $K[J_{\infty}(V')]$, and 
$$K[J_{\infty}(V')]^{J_{\infty}(G)} = K[J_{\infty}(V)]^{J_{\infty}(G)} \oplus \cI^{J_{\infty}(G)}.$$ If $i>p$ or $j>p$, $X^{(k)}_{ij}\in {\cI}^{J_{\infty}(G)}$, so $$K[J_{\infty}(V)]^{J_{\infty}(G)} \cong K[J_{\infty}(V')]^{J_{\infty}(G)} \slash{\cI}^{J_{\infty}(G)}$$ is generated by $X^{(k)}_{ij}$, $1\leq i,j\leq p$. Therefore $K[J_{\infty}(V)]^{J_{\infty}(G)} = \text{Im}(Q^K_h)$, as claimed. \end{proof}

\begin{proof}[Proof of Theorem \ref{thm:main}]
	By Theorems \ref{thm:injectiveGL} and \ref{thm:JGLinvariant}, $K[J_{\infty}(V)^{J_\infty(Sp_h(K))}=\text{Im}(Q^K_h) \cong K[J_{\infty}(Pf_{h+2})]$.
\end{proof}

\begin{proof}[Proof of Corollary \ref{cor:quot}] This is immediate from Theorem \ref{thm:main} because $V /\!\!/ Sp_h(K)$ is isomorphic to the Pfaffian variety $Pf_{h+2}$.\end{proof}

\section{Some properties of standard monomials }
By the definition of standard monomials, if $E_1E_2\cdots E_n\in \cS\cM(\cE)$, then $E_{i+1}$ is the largest element in $||\cE(E_{i+1})||$ such that $E_i\leq E_{i+1}$. In this section, we study the properties of $||\cE(E_{i+1})||$ and $E_{i+1}$ that need to be satisfied to make $E_1E_2\cdots E_n$ a standard monomial.

Let 
$$E=|(u_h,k_h),\dots,(u_1,k_1)|\in \cE,$$
$$ 
J'=\bpartial^{n'}|u'_{h'},\dots,u'_1|\in \cJ.$$
 
\subsection{ $L(E,J')$}

 For $h'\leq h$,
let $\sigma$ be the permutation of $\{1,2,\dots,h'\}$  such that $u_{\sigma(i)}<u_{\sigma(i+1)}$.
Let $L(E,J')$ be the smallest non-negative integer $i_0$ such that 
$u'_{i}\geq u_{\sigma(i-i_0)}$, $i_0< i \leq h'$.
Let
 $$E(h')=|(u_{h'},k_{h'}),\dots,(u_1,k_1)|.$$
Then $L(E,J')=L(E(h'),J')$.

The following lemma is obvious.
\begin{lemma}\label{lemma:replace}
For $J''=\bpartial^k|u''_{h'},\dots,u''_1|\in\cJ$, if there are at least $s$ elements in $\{u''_{h'},\dots, u''_1\} $ from the set  $\{u'_{h'},\dots, u'_1\} $, then 
$L(E,J'')\geq L(E,J')-h'+s$;  
\end{lemma}

\subsection{A criterion for $J'$ to be greater than $E$}
We say
$J'$ is greater than $E$ if there is an element $E'\in\cE(J')$ with $E\leq E'$. Then $J'$ is greater than $E$  if and only if $J'$  is greater than $E(h')$. The following lemma is a criterion for $J'$ to be greater than $E$.
\begin{lemma} \label{lemma:critgreat}
		$J'$ is greater than $E$ if and only if 
	$wt(J') - wt(E(h')) \geq L(E,J')$.
\end{lemma}
\begin{proof} 
	Let
	$i_0=L(E,J')$ and $\sigma$ be the permutation of $\{1,2,\dots,h'\}$ such that $u_{\sigma(i)}<u_{\sigma(i+1)}$ .

	If $wt(J') - wt(E(h')) \geq L(E,J')$,	
	let $$\tilde u'_{\sigma(i)}= \left\{ \begin{array}{cc}
	u'_{i+i_0}, & \sigma(i)+i_0\leq h' \\
	u'_{i+i_0-h_b},&   i+i_0> h'
	\end{array} \right. ,
	\quad \quad k'_{\sigma(i)}= \left\{  \begin{array}{cc}
	k_{\sigma(i)}, & i+i_0\leq h', i\neq h'\\
	k_{\sigma(i)}+1,&   i+i_0> h', i\neq h'
	\end{array} \right. ,
	$$
	$$k'_{\sigma(h')}=wt(J')-\sum_{i=1}^{h'-1}k'_{\sigma(i)}.$$
	Then
	$$k'_{\sigma(h')}=wt(J')-wt(E(h'))-i_0+k_{\sigma(h')}+1-\delta_{i_0}^{0}\geq k_{\sigma(h')}+1-\delta_{i_0}^{0}.$$
	$$(\tilde u'_{\sigma(i)},k'_{\sigma(i)})\geq (u_{\sigma(i)},k_{\sigma(i)}). $$
	So $$\tilde E'=|(\tilde u'_{h'},k'_{h'}),\dots,(\tilde u'_{2},k'_2),(\tilde u'_{1},k'_1)|$$
	is an element in $\cE(J')$ with $\tilde E'\geq E$.
	
	On the other hand, suppose $\tilde E'\in \cE(J')$ with $\tilde E'\geq E$.
	Assume $$\tilde E' =|(\tilde u'_{h'},k'_{h'}),\dots,(\tilde u'_{2},k'_2),(\tilde u'_{1},k'_1)|.$$
	We have
	$(\tilde u'_i,k'_i)\geq (u_i,k_i)$ i.e.,
	$k'_i>k_i$ or $k'_i=k_i, \tilde u'_i\geq u_i$.	
	So
	$$\sum_{i=1}^{h'}(k_i'-k_i)+\sharp\{\tilde u'_i\geq u_i, i|1\leq i\leq h'\}\geq h', $$
	Let $i'_0=h'-\sharp\{\tilde u'_i\geq u_i, i|1\leq i\leq h'\}$. Then 
	$$i'_0\leq \sum_{i=1}^{h'}(k_i'-k_i)= wt(J')-wt(E(h')).$$
	Here $\tilde u_1',\dots,\tilde u'_{h'}$ is a permutation of $u'_1,\dots,u'_{h'}$.
	By the definition of $i'_0$, it is easy to see that  $u'_{i}\geq u_{\sigma(i-i'_0)}$, $i'_0< i \leq h'$. So $i'_0\geq L(E,J')$. Thus
	$$ wt(J') - wt(E(h'))\geq i_0'\geq L(E,J').$$
	
\end{proof}

\begin{cor} \label{cor:critgreat1}
	$J'$ is greater than $E$ if and only if $||E(h')||J'$ is standard.
\end{cor}
\begin{proof}
	By Lemma \ref{lemma:critgreat}, $J'$ is greater than $E$ if and only if $wt(J') - wt(E(h')) \geq L(E,J')$ and $||E(h')||J'$ is standard if and only if $wt(J') - wt(E(h')) \geq L(E(h'),J')=L(E,J')$.
\end{proof}

\subsection{The property \lq\lq largest"}	
Let
\begin{equation*}
\cW_s(E,J')=\{J=\bpartial^k|u'_{i_s}\dots u'_{i_1}| \ |\  1\leq i_l\leq h',\ J \text{ is greater than } E\}.
\end{equation*}

\begin{lemma}\label{cor:compare}
If $E'$ is the largest element in $\cE(J')$ such that $E\leq E'$, then for $s<h'$,
 $||E'(s)||$ is the smallest element in $\cW_s(E,J')$.
	\end{lemma}
\begin{proof}
	Assume $$E'=|( u'_{h'},k'_{h'}),\dots,(u'_{2},k'_2),( u'_{1}, k'_1)|.$$
	For $s<h'$, let $J_s$ be the smallest element in $\cW_s(E,J')$. Let
	$$E_s=|( u'_{i_s},\tilde k_{s}),\dots,(u'_{i_2},\tilde k_2),( u'_{i_1},\tilde k_1)|$$
	 be the largest element in $\cE(J_s)$ such that $E(s)\leq E _s$. 
	 
	 Assume $l$ is the largest number such that $(u'_j, k'_j)=(u'_{i_j},\tilde k_j)$ for $j<l\leq s+1$. If $l\leq s$, then $i_l\geq l$ and $(u'_{i_l},\tilde k_l)\neq (u'_l,k'_l)$. If $i_l=l$, by the maximality of $E'$ and the minimality of $J_s$, we must have $(u'_{i_l},\tilde k_l)= (u'_l,k'_l)$, a contradiction. So $i_l>l$.
	 
	 If $(u'_{i_l},\tilde k_l)<(u'_l,k'_l)$, then
	$(u'_{l}, k'_l+k'_{i_l}-\tilde k_{l})> (u'_{i_l}, k'_{i_l})$.
	Let $E''$ be the element in $\cE(J')$ obtained by replacing $( u'_{l},k'_l)$ and $(u'_{i_l},k'_{i_l})$ in $E'$
	by $(u'_{i_l},\tilde k_{l})$ and   $(u'_{l}, k'_l+k'_{i_l}-\tilde k_{l})$, respectively. We have $E'\prec E''$ and $E(h')\leq E''$.
	But $E'\neq E''$ is the largest element in $\cE(||E'||)$ such that $E\leq E'$, which is a contradiction.

	Assume $(u'_{i_l},\tilde k_l)>(u'_l,k'_l)$. If $l\notin\{i_1,\dots,i_s\}$, replacing $(u'_{i_l},\tilde k_l)$ in $E_s$ by $(u'_l,k'_l)$, we get $E'_s$ with $E(s)\leq E'_s$ and $||E'_s||\prec J_s$. This is impossible since $J_s\neq ||E'_s||$
	is the smallest element in $\cW_s(E, ||E'||) $. If $l=i_j \in\{i_1,\dots,i_s\}$,
	$(u'_{i_l}, \tilde k_l+\tilde k_{j}- k'_{l})> (u'_{i_j}, \tilde k_{j})$.
	Let $E'_s$ be the element in $\cE(J_s)$ obtained by replacing $(u'_{i_l},\tilde k_l)$ and $(u'_{i_j}, \tilde k_{j})$ in $E_s$
	by $(u'_l,k'_l)$ and   $(u'_{i_l}, \tilde k_l+\tilde k_{j}- k'_{l})$, respectively. We have $E_s\prec E'_s$ and $E\leq E'_s$.
	But $E'_s\neq E_s$ is the largest element in $\cE(J_s)$ such that $E\leq E'_s$, a 
	contradiction. Therefore $E_s=E'(s)$, and $||E'(s)||$ is the smallest element in $\cW_s(E,J')$.
\end{proof}
\begin{cor}\label{cor:lrnumber1}
If $E'$ is the largest element in $\cE(||E'||)$ such that $E\leq E'$, then for $s<h'$,
$$L(E,||E'(s)||)=wt(E'(s))-wt(E(s)).$$	
\end{cor}
\begin{proof}
Since $E\leq E'$, $E\leq E'(s)$. 
By Lemma \ref{cor:compare}, $||E'(s)||$ is the smallest element in $\cW_s(E,J')$. By Lemma \ref{lemma:critgreat}, $L(E,||E'(s)||)=wt(E'(s))-wt(E(s)).$
\end{proof}
\begin{cor}\label{cor:lrnumber2}
If $E'$ is the largest element in $\cE(||E'||)$ such that $E\leq E'$, then for $s<h'$ and any $J\in \cW_s(E,||E'||)$,
	$$L(E,||E'(s)||)\leq L(E,J).$$
\end{cor}
\begin{proof}
 Assume $J=\bpartial^k|u_s,\dots, u_1|$ and  $||E'(s)||=\bpartial^l|u'_s,\dots, u'_1|$. If $m=L(E,J)-L(E,||E'(s)||)<0$, let $J''= \bpartial^{l+m}|u_s,\dots, u_1|$. By Lemma \ref{lemma:critgreat},  $J''\in \cW_s(E,J')$.
	By Lemma \ref{cor:compare}, $||E'(s)||$ is the smallest element in $\cW_s(E,||E'||)$. But $wt(||E'(s)||)>wt(J'')$, a contradiction.
	\end{proof}

\begin{lemma}\label{lem:lrnumber}
	Let $E_i=|(u^i_{h_i},k^i_{h_i}),\dots,(u^i_1,k^i_1)|$, $i=a,b$. Suppose that $E_b\leq E_a$, and that $E_a$ is the largest element in $\cE(||E_a||)$ such that $E_b\leq E_a$. Let $1\leq h<h_a$ and $\sigma_i$ be permutations of $\{1,\dots, h\}$, such that $u^i_{\sigma_i(1)}<u^i_{\sigma_i(2)}<\cdots<u^i_{\sigma_i(h)}$.   Let $u_1',\dots,u'_{h_a}$ be a permutation of $u^a_1,\dots,u^a_{h_a}$ such that $u_1'<u_2'<\cdots<u'_{h_a}$.
	Assume $u'_{i_2}=u^a_{\sigma(i_1)}$ with $i_2>i_1$. Then for any
		$$K=\bpartial^k|u_h,\dots, u_{s+1},u'_{t_s},\dots u'_{t_1}|,$$
		with $t_1<t_2<\cdots <t_s<i_2$, we have $L(E_b,K)>L(E_b,||E_a(h)||)+s-i_1$. 
\end{lemma}

\begin{proof}
	 Let $n=wt(E_a(h))$.\\
	Let $i_0=L(E_b,K)$, then
	$u'_{t_i}\geq u^{b}_{\sigma_{b}(i-i_0)}$, for $s\geq i>i_0$.\\
	 Let $i'_0= L(E_b, ||E_a(h)||)$, then
	$u^a_{\sigma_a(i)}\geq u^{b}_{\sigma_{b}(i-i'_0)}$, for $h\geq i>i'_0$.

	We have $i_1>i'_0$.
	Otherwise, $i_1\leq i'_0$. Replacing $u^a_{\sigma_a(i_1)}$ in $||E^a(h)||$ by some $u^a_i<u^a_{\sigma_a(i_1)}$ with $i>h$, (such $u^a_i$ exists since $i_2>i_1$), we get  $$J=\bpartial^{n}|u_{\sigma'_{a}(h)}^{a},\dots,u_{\sigma'_{a}(i_1+1)}^{a}\, u^a_i,u_{\sigma'_{a}(i_1-1)}^{a},\dots, u^{a}_{\sigma'_{a}(1)}|
	$$ with
	$L(E_b,J)\leq i'_0.$
	By Lemma \ref{lemma:critgreat}, $J$ is greater than $E_b$.
	It is impossible by Lemma \ref{cor:compare} since 
    $J\prec ||E_a(h)||$ and 
   $E_a$ is the largest element in $\cE(||E_a||)$ such that $E_b\leq E_a$.

	If 
	$s\geq i_1$, let $$J'= \bpartial^{n}|u_{\sigma_{a}(h)}^{a},\dots,u_{\sigma_{a}(i_1+1)}^{a},  u'_{t_s}, \dots,u'_{t_{s-i_1+1}}|.$$
 If $L(E_b,K)\leq L(E_b,||E_a(h)||)+s-i_1$,
$$u'_{t_i}\geq u^{b}_{\sigma_{b}(i-i_0)}\geq u^{b}_{\sigma_{b}(i-i'_0-s+i_1)}, \,\,\text{for }i\leq s.$$	
	 We have
	$L(E_b,J')\leq i'_0.$ By Lemma \ref{lemma:critgreat}, $J'$ is greater than $E_b$. By Lemma \ref{cor:compare}, it is impossible since 
	$J'\prec ||E_a(h)||$ and $E_a$ is the largest element in $\cE(||E_a||)$ such that $E_b\leq E_a$.
	So 	when $s\geq i_1$, $L(E_b,K)\geq L(E_b,||E_a||)+s-i_1$.

	 If $s<i_1$, let $t'_1<t'_2<\cdots<t'_{i_1}<i_2$ with $\{t_1,\dots, t_s\}\subset \{t'_1,\dots, t'_{i_1}\}$. Let $$K'=\bpartial^k|u'_{t'_1},\dots,u'_{t'_{i_1}},u_{i_1+1},\dots,u_h|.$$
	By Lemma \ref{lemma:replace}, we have $$L(E_b,K)\geq L(E_b,K')+s-i_1
	\geq  L(E_b,||E_a(h)||)+s-i_1.$$
	
	So for any $s>0$, we have $L(E_b,K)> L(E_b,||E_a(h)||)+s-i_1$.
\end{proof}
The following lemmas are obvious.
\begin{lemma}\label{lemma:order1}If $J_1\prec J_2\prec \cdots \prec J_n$, $\sigma$ is a permutation of $\{1,\dots ,n\}$, then
	$$J_1J_2\cdots J_n\prec J_{\sigma(1)}\cdots J_{\sigma(n)}.$$
\end{lemma}
\begin{lemma} \label{lemma:order2}If $K_1K_2\cdots K_k \prec J_1\cdots J_l$, then
	$$K_1K_2\cdots K_{s-1}JK_s\cdots K_k \prec J_1\cdots J_{s-1}JJ_s\cdots J_l.  $$
\end{lemma}

\section{Proof of Lemma \ref{lem:base0}}\label{section:proof2.5}
In this section we prove Lemma \ref{lem:base0}.
By Lemma \ref{lemma:order1}, we can assume the monomials are expressed as ordered products $J_{1}J_{2}\cdots J_{b}$ with $J_{a}\prec J_{a+1}$.
For $\alpha \in \cM(\cJ)$,
let $$\R(\alpha)=\{\sum c_i\beta_i\in\R|\ c_i\in\mathbb Z, \ \beta_i\in \cM(\cJ), \ \beta_i\prec \alpha, \ \beta_i\neq \alpha\},$$
the space of linear combinations of elements preceding $\alpha$ in $\cM(\cJ)$ with integer coefficients. 
\begin{lemma}\label{lemma:case2}
	If $J_1J_2$ is not standard, $J_1J_2\in \R(J_1)$. 
\end{lemma}
\begin{proof}
	Assume $J_i=\bpartial^{n_i}|u^i_{h_i},\dots,u^i_2,u^i_1|$, for $i=1,2$. \\
	Let $E_1=|(u^1_{h_1},n_1),\dots,(u^1_1,0)|$.
Let $i_0=L(E_1,J_2)$.
 If $i_0\neq 0$, there is $i_0\leq i_1\leq h_2$, such that $u^2_{i_1}<u^1_{i_1-i_0+1}$. If $i_0=0$, let $i_1=0$.
Let  $m=n_1+n_2$ and $a_{n_2-l}=\delta^0_l$, $0\leq l\leq l_0-1$.
By Lemma \ref{lem:fullrelation0}, there are integers $a_k$ such that
	 	\begin{equation}\label{eqn:relation110}
	 	\sum \epsilon a_{k} 
	 	\bpartial^{m-k} |\underline{u^1_{h_1},\dots,u^1_{i_1-i_0+1}},u^1_{i_1-i_0},\dots,u^1_1|
	 	\bpartial^{k} |u^2_{h_2},\dots,u^2_{i_1+1},\underline{u^2_{i_1},\dots,u^2_1}|
	 	\in \R[h_1+2].
	 	\end{equation}
	 
	 		\begin{enumerate}
	 		\item If $h_1=h_2$, then $n_1\leq n_2$.
	 		Since $J_1J_2$ is not standard, $J_2$ is not greater than $E_1$. By Lemma \ref{lemma:critgreat},  $i_0> n_2-n_1\geq 0$.
	 	$J_1J_2\in \R(J_1)$ since in Equation (\ref{eqn:relation110}):
		\begin{itemize}
			\item
			All the terms with $k=n_2$ precede $J_1$ except $J_1J_2$ itself;
			\item
			All the terms with $k=n_2-1,\dots,n_1$ vanish since $a_{k}=0$;
			\item
			All the terms with $k=n_2+1,\dots, m$ precede $J_1$ since the weight of the upper $\bpartial$-list is $m-k<n_1$;
			\item
			All the terms with $k=0,\dots, n_1-1$ precede $J_1$ after exchanging the upper $\bpartial$-list and the lower $\bpartial$-list since the weight of the lower $\bpartial$-list is $k<n_1$;
			\item
			The terms in $\R[h_1+2]$ precede $J_1$ since they have bigger sizes.
		\end{itemize}
		
		\item Suppose $h_1>h_2$.
		Since $J_1J_2$ is not standard, by Lemma \ref{lemma:critgreat}, $i_0> n_2$.
		$J_1J_2\in \R(J_1)$ since in Equation (\ref{eqn:relation110}):
		\begin{itemize}
			\item
			All the terms with $k=n_2$ precede $J_1$ in the lexicographic order except $J_1J_2$ itself;
			\item
			All the terms with $k=n_2-1,\dots,0$ vanish since $a_{k}=0$;
			\item
			All the terms with $k=n_2,\dots, m$ precede $J_1$ since the weight of the upper $\bpartial$-list is $m-k<n_1$;
			\item
			The terms in $\R[h_1+2]$ precede $J_1$ since they have bigger sizes.
		\end{itemize}
	\end{enumerate}
\end{proof}

\begin{proof}[Proof of Lemma \ref{lem:base0}]
	We prove the lemma by induction on $b$.\\
	If $b=1$, $J_1$ is standard.\\
	If $b=2$, by Lemma \ref{lemma:case2}, the lemma is true. \\
	For $b\geq 3$, assume the lemma is true for $b-1$.
 We can assume $J_1\cdots J_{b-1}$ is standard by induction and Lemma \ref{lemma:order2}. Let $E_1\cdots E_{b-1}\in\cS\cM(\cE)$ be the standard ordered product of elements of $\cE$ corresponding to $J_1\cdots J_{b-1}$. If $J_1\cdots J_b$ is not standard, then $J_b$ is not greater than $E_{b-1}$. By Lemma \ref{lemma:straight1} (below), $J_{b-1}J_b=\sum K_if_i$ with $K_i\in \cJ$, $f_i\in \R$ such that $K_i$ is either smaller than $J_{b-1}$ or $K_i$ is not greater than $E_{b-2}$. If $K_i$ is smaller than $J_{b-1}$, then $J_1\cdots J_{b-1}K_if_i \in \R(J_1\cdots J_{b-1})$. 
  If $K_i$ is not greater than $E_{b-2}$, $J_1\cdots J_{b-2}K_i$ is not standard, so it is in $\R(J_1\cdots J_{b-2})$ by induction. Then $J_1\cdots J_{b-2}K_if_i \in \R(J_1\cdots J_{b-1})$. So $J_1\cdots J_b=\sum J_1\cdots J_{b-2}K_if_i\in \R(J_1\cdots J_{b-1})$.
\end{proof}
\begin{lemma}\label{lemma:straight1}
	Let $E\in \cE$, $J_a$ and $J_b$ in $\cJ$ with $J_a\prec J_b$, and suppose that $E_a$ is the largest element in $\cE(J_a)$ such that $E\leq E_a$. If $J_b$ is not greater than $E_a$, then $J_aJ_b=\sum K_if_i$ with $K_i\in \cJ$, $f_i\in \R$ such that $K_i$ is either smaller than $J_a$ or $K_i$ is not greater than $E$. 
\end{lemma} 
 \begin{proof}
 	Assume
 	$J_a=\bpartial^{n_a}|u'_{h_a},\dots,u'_{1}|$ and $J_b=\bpartial^{n_b}|u^b_{h_b},\dots,u^b_{1}|$.
 	$J_a\prec J_b$, so $h_a\geq h_b$.
 	Assume
$||E_a(h_b)||=\bpartial^{m_a}|u^a_{h_b},\dots,u^a_{1}|$.
Let  $m=n_b+n_a$.
	Let $i_0=L(E_a, J_b)$.
 If $i_0\neq 0$, there is $i_0\leq i_1\leq h_b$, such that $u^b_{i_1}<u^a_{i_1-i_0+1}$. If $i_0=0$, let $i_1=0$.
Since $J_b$ is not greater than $ E_a$, by Lemma \ref{lemma:critgreat},
 \begin{equation}\label{eqn:numberl0} i_0> n_b-m_a.
 \end{equation}
 By definition, $\{u^a_{h_b},\dots,u^a_{1}\}$ is a subset of $\{u'_{h_a},\dots, u'_{1}\}$ with
 $u'_{i}<u'_{i+1}$ and $u^a_{i}<u^a_{i+1}$. If we assume $u'_{i_2}=u^a_{i_1-i_0+1}$, we have $i_2\geq i_1-i_0+1$.

Now we prove the lemma. The proof is quite long and it is divided into three cases.
	
	\textsl{Case 1:} $h_a=h_b$. Let $a_{n_b-l}=\delta^0_l$ for $0\leq l\leq i_0-1$. 
	By Lemma \ref{lem:fullrelation0}, there are integers $a_k$, such that 
	\begin{equation}\label{eqn:relation111}
	\sum \epsilon a_{k}
	\bpartial^{m-k} |\underline{u'_{h_a},\dots,u'_{i_1-i_0+1}},u'_{i_1-i_0},\dots,u'_1|\,\,
	\bpartial^{k} |u^b_{h_b},\dots,u^b_{i_1+1},\underline{u^b_{i_1},\dots,u^b_1}|
	\in \R[h_a+2].
	\end{equation}
	$J_aJ_b\in \R(J_a)$ since in the above equation, 
	\begin{itemize}
		\item
		All the terms with $k=n_b$ precede $J_a$ in the lexicographic order except $J_aJ_b$ itself;
		\item
		All the terms with $k=n_b-1,\dots,n_a$ vanish since $n_a=m_a$, $a_{k}=0$;
		\item
		All the terms with $k=n_b+1,\dots, m$ precede $J_a$ since the weight of the upper $\bpartial$-list is $m-k<n_a$;
		\item
		All the terms with $k=0,\dots, n_a-1$ precede $J_a$ after exchanging the upper $\bpartial$-list and the lower $\bpartial$-list since  the weight of the lower $\bpartial$-list is $k<n_a$;
		\item
		The terms in $\R[h_a+2]$ precede $J_a$ since they have bigger sizes.
	\end{itemize}

	\textsl{Case 2:} $h_b<h_a$ and $n_b<m_a$. 
	
	By Lemma \ref{lem:fullrelation0}, \begin{equation}\label{eqn:relationterms} 
	\sum_{0\leq i< h_b}\sum_{\sigma}\frac {(-1)^{i}\sign(\sigma)}{i!(h_b-i)}\sum \epsilon a^{i}_{k} 	\end{equation}
	$$	
	\bpartial^{m-k}|\underline{u'_{h_a},\dots,u'_{i+1}},u^b_{\sigma{(i)}},\dots,u_{\sigma(2)}^b,u^b_{\sigma(1)}|\,\,
	\bpartial^k|\underline{u_{\sigma(h_b)}^b,\dots,u^b_{\sigma{(i+1)}},u'_{i},\dots,u'_{1}}|\in \R[h_a+2].
$$
	
	Here  $a_k^{i}$ are integers and $a_{n_b-l}^{i}=\delta_{0,l}$ for $0 \leq l<h_b -i$. The second summation is over all permutations $\sigma$ of $\{1,\dots,h_b\}$.
	In equation \eqref{eqn:relationterms}:
	\begin{itemize}
		\item
		The terms in $\R[h_a+2]$ precede $J_a$ since they have bigger sizes.
		\item
		All the terms with $k=n_b,\dots, m$ precede $J_a$ since the weight of the upper $\bpartial$-list is $m-k<n_a$.
		\item
		The terms with $k=n_b$ are $J_aJ_b$
		and the terms with the lower $\bpartial$-lists $$K_0=\bpartial^{n_b}|u'_{i_{h_b}},\dots,u'_{i_2},u'_{i_1}|\in \cJ.$$
		All of the other terms cancel.
		By Corollary \ref{cor:lrnumber1} and \ref{cor:lrnumber2}, 
		$$L(E,K_0)\geq L(E,E_{a}(h_b))=m_a- wt(E(h_b)) >n_b- wt(E(h_b)).$$
		By Lemma \ref{lemma:critgreat}, $K_0$ is not greater than $E$.
		\item
		The terms with $k<n_b$ vanish unless $h_b-i\leq  n_b-k$. In this case the lower $\bpartial$-lists of the terms are
		$$K_1=\bpartial^k|u_{\sigma(h_b)}^b,\dots,u^b_{\sigma{(i+1)}},u'_{s_i},\dots,u'_{s_1}|.$$
	    By Lemma \ref{lemma:replace},
		$$L(E,K_1)\geq L(E,K_0)-(h_b-i).$$
		So
		$$L(E,K_1)>n_b-wt(E(h_b)) -(h_b-i) \geq k-wt(E(h_b)).$$	
		By Lemma \ref{lemma:critgreat}, $K_1$ is not greater than $E$. 
	\end{itemize}

	\textsl{Case 3:}
	$h_a>h_b$ and $n_b\geq m_a$.
 If $i_0=0$, then $J_b$ is greater than $E_a$ and $J_1\cdots J_b$ is standard.
		So $i_0> 0$.
		By Lemma \ref{lem:fullrelation0}, we have
		\begin{equation}\label{eqn:relationterms51}
		\sum_{ \min\{i_1, i_2\}>s}\sum_{\sigma}\frac {(-1)^{s}\sign(\sigma)}{s!(i_2-1-s)!}\sum_{k=0}^{m} \epsilon a^{s}_{k} 	\end{equation}
		$$
		\bpartial^{m-k}|\underline{u'_{h_a} ,\dots,u'_{i_2} ,u'_{\sigma(i_2-1)},\dots, u'_{\sigma(s+1)},u_{s}^b,\dots,u^b_{1}}|\,\,
		\bpartial^k|u^b_{h_b},\dots,u^b_{i_1+1},\underline{u^b_{i_1},\dots,u^b_{s+1}},u'_{\sigma(s)},\dots,u'_{\sigma(1)}|
	$$
		$$  \in \R[h_a+2].$$
		Here $a_k^{s}$ are integers, $a_{n_b-l}^{s}=\delta_{0,l}$
		for $0 \leq l<i_1-s$, and $\sigma$ are permutations of $\{1,\dots,i_2-1\}$. Similarly, by Lemma \ref{lem:fullrelation0}, we have

	\begin{equation}\label{eqn:relationterms53}
	\sum_{\substack{i\geq i_1> s\\ i_2>s}}\sum_{\sigma,\sigma_1}\frac {(-1)^{i+s+t}\sign(\sigma)\sign(\sigma_1)}{(i-i_1)!(h_b-i)!s!(i_2-1-s)!}\sum_{k=0}^{m}\epsilon a^{i,s}_{k} 	
	\end{equation}
		$$
		\bpartial^{n-k}|u^b_{\sigma(h_b)}\dots\underline{u^b_{\sigma(i)}\dots u^b_{\sigma(i_1+1)},u^b_{i_1} \dots u^b_{1},u'_{h_a} \dots  u'_{h_b+1}}
	|\,
		\bpartial^k|\underline{u'_{h_b+1},\dots u'_{i_2},u'_{\sigma_1(i_2-1)} ,\dots u'_{\sigma_1(s+1)}},u'_{\sigma_1(s)}, \dots u'_{\sigma_1(1)}|
		$$
		$$  \in \R[h_a+2].$$
		Here $a_k^{i,s}$ are integers, $a_{n_b-l}^{i,s}=\delta_{0,l}$ for $0 \leq l<(i+j_1-s)$, 
		$\sigma$ are permutations of $\{h_b,\dots, i_1+1\}$, and
		$\sigma_1$ are permutations of $\{1,\dots,i_2-1\}$.
		\begin{itemize}
			\item We use Equation (\ref{eqn:relationterms51}) when $i_2>i_1-i_0+1$;
			\item We use Equation (\ref{eqn:relationterms53}) when $i_2=i_1-i_0+1$. 
		\end{itemize}
		In the above relations, we have the following:
		\begin{enumerate}
			\item
			The terms in $\R[h_a+2]$ precede $J_a$ since they have bigger sizes.
			\item
			All of the terms with $k=n_b+1,\dots, n$ precede $J_a$ since the weight of the upper $\bpartial$-list is $n-k<n_a$.
			\item
			The terms with $k=n_b$ are $J_aJ_b$, the terms with upper $\bpartial$-list preceding $J_a$ (the upper $\bpartial$-lists are the $\bpartial$-lists given by replacing some $u'_{i}$, $i\geq i_2$ in $J_a$ by some $u^b_{k}$), and the terms with lower $\bpartial$-list
			$$K_0=\bpartial^{n_b}|u_{h_b}^b,\dots,u_{i_1+1}^b, u'_{\sigma_1(i_1)},\dots u'_{\sigma_1(1)}|.$$
			All of the other terms cancel.
			By Lemma \ref{lem:lrnumber}, 
			$$L(E,K_0)>L(E,||E_a(h)||)+i_1-(i_1-i_0+1);$$
			By the above inequality,
			\begin{eqnarray*}L(E,K_0)&\geq & L(E,||E_a(h_b)||)+i_0\\
				(\text{by Corollary \ref{cor:lrnumber1}})\quad\quad
				&=&wt(E_a(h_b))-wt(E)+i_0\\
			(\text{by Equation (\ref{eqn:numberl0}})\quad\quad	&>&wt(E_b)-wt(E(h_b)).
			\end{eqnarray*}

			By Lemma \ref{lemma:critgreat}, $K_0$ is not greater than $E$.
	
			\item When $i_2>i_1-i_0+1$,
			the terms with $k<n_b$ in Equation (\ref{eqn:relationterms51}) vanish unless $i_1-s\leq  n_b-k$. In this case the lower $\bpartial$-lists of the terms are
			$$K_1=\bpartial^k|u^b_{h_b},\dots,u^b_{i_1+1},\underline{u^b_{i_1},\dots,u^b_{s+1}},u'_{\sigma(s)},\dots,u'_{\sigma(1)}|.$$
			The underlined $u$ in $K_1$ can be any underlined $u$ in Equation (\ref{eqn:relationterms51}). By Lemma \ref{lem:lrnumber}
			$$L(E,K_1)> L(E,E_a(h_b))+s-(i_1-i_0+1);$$
			\begin{eqnarray*}
				L(E,K_1)&\geq& L(E,E_a(h_b))+k-n_b+i_0\\
				(\text{by Corollary \ref{cor:lrnumber1}}) &=&wt(E_a(h_b))-wt(E(h_b))+k-n_b+i_0\\
			(\text{by Equation (\ref{eqn:numberl0}})\,	&>& k-wt(E(h_b)).
			\end{eqnarray*}
			By Lemma \ref{lemma:critgreat}, $K_1$ is not greater than $E$.
			
			\item When $i_2=i_1-i_0+1$, the terms with $k<n_b$ are the terms in the Equation (\ref{eqn:relationterms53}), such that $i-s\leq  n_b-k$. In this case, the lower $\bpartial$-lists of the terms are
			$$K_1=\bpartial^k|\underline{u'_{h_b+1},\dots, u'_{i_2},u'_{\sigma_1(i_2-1)} ,\dots ,u'_{\sigma_1(s+1)}},u'_{\sigma_1(s)} ,\dots ,u'_{\sigma_1(1)}|.$$
			The underlined $u$ in $K_1$ can be any underlined $u$ in Equation (\ref{eqn:relationterms53}).
			Let 
			$$
			K'_1=\bpartial^k|u'_{k_{h_b}},\dots,
			u'_{k_{i_2}},u'_{\sigma(i_2-1)},\dots,u'_{\sigma(1)}|.$$
			Here $i_2\leq k_{i_2}< k_{i_2+1}<\cdots< k_{h_b}\leq h_a$. By Lemma \ref{lemma:replace}, there is some $K'_1$ such that 
			\begin{eqnarray}
			\label{eqn:RK1-3}
			L(E,K_1)\geq L(E,K'_1)-(i_2-1-s)-(i-i_1),
			\end{eqnarray}
			since in $K_1$ the number of $u^b_l$ with $u^b_l>u'_{i_2}$ is at most $i-i_1$. By Corollary \ref{cor:lrnumber2},
			\begin{equation}\label{eqn:LRK1-3+}
	 L(E,K_1')\geq L(E,||E_a(h_b)||).
			\end{equation}
			By Equations (\ref{eqn:RK1-3}, \ref{eqn:LRK1-3+}),
			\begin{eqnarray*}
				L(E,K_1)&\geq& L(E,||E_a(h_b)||)+s-(i-i_0)\\
					(\text{by Corollary \ref{cor:lrnumber1}})&\geq& wt(E_a(h_b))-wt(E(h_b))+k-n_b+i_0\\
				(\text{by Equation \ref{eqn:numberl0}})	&>& k-wt(E(h_b)).
			\end{eqnarray*}
			By Lemma \ref{lemma:critgreat}, $K_1$ is not greater than $E$.
		\end{enumerate}

	\end{proof}

\end{document}